\documentclass{my_aims}

\usepackage{amsmath}
\usepackage{subfig}
\usepackage{graphicx}

\usepackage{paralist}
\usepackage[colorlinks=true]{hyperref}

\hypersetup{urlcolor=blue,citecolor=red}

\newtheorem{theorem}{Theorem}[section]
\newtheorem{corollary}{Corollary}
\newtheorem{lemma}[theorem]{Lemma}
\newtheorem{proposition}[theorem]{Proposition}
\theoremstyle{definition}

\newtheorem{remark}{Remark}

\def\R{\mathbb{R}}

\title[A TB-HIV/AIDS coinfection model]{A TB-HIV/AIDS coinfection model and optimal control treatment}

\author[C. J. Silva and D. F. M. Torres]{}

\subjclass{Primary: 92D30, 93A30; Secondary: 34D30, 49J15.}

\keywords{Tuberculosis, human immunodeficiency virus,
coinfection, treatment, equilibrium, stability, optimal control.}

\email{cjoaosilva@ua.pt}
\email{delfim@ua.pt}

\begin{document}

\maketitle


\centerline{\scshape Cristiana J. Silva and Delfim F. M. Torres}
\medskip
{\footnotesize
\centerline{Center for Research and Development in Mathematics and Applications (CIDMA)}
\centerline{Department of Mathematics, University of Aveiro, 3810--193 Aveiro, Portugal}
}


\begin{abstract}
We propose a population model for TB-HIV/AIDS coinfection transmission dynamics,
which considers antiretroviral therapy for HIV infection and treatments for latent
and active tuberculosis. The HIV-only and TB-only sub-models are analyzed separately,
as well as the TB-HIV/AIDS full model. The respective basic reproduction numbers are computed,
equilibria and stability are studied. Optimal control theory is applied to the TB-HIV/AIDS model
and optimal treatment strategies for co-infected individuals with HIV and TB are derived.
Numerical simulations to the optimal control problem show that non intuitive measures
can lead to the reduction of the number of individuals with active TB and AIDS.
\end{abstract}


\section{Introduction}

According with the World Health Organization (WHO),
the human immunodeficiency virus (HIV) and mycobacterium tuberculosis are the first
and second cause of death from a single infectious agent, respectively \cite{TB_WHO_report_2013}.
Acquired immunodeficiency syndrome (AIDS) is a disease of the human immune system caused
by infection with HIV. HIV is transmitted primarily via unprotected sexual intercourse,
contaminated blood transfusions, hypodermic needles, and from mother to child during pregnancy,
delivery, or breastfeeding \cite{Book:Env:Medicine}. There is no cure or vaccine to AIDS.
However, antiretroviral (ART) treatment improves health, prolongs life, and substantially
reduces the risk of HIV transmission. In both high-income and low-income countries,
the life expectancy of patients infected with HIV who have access to ART is now measured
in decades, and might approach that of uninfected populations in patients who receive
an optimum treatment (see \cite{AIDS:chronic:Lancet:2013} and references cited therein).
However, ART treatment still presents substantial limitations: does not fully restore health;
treatment is associated with side effects; the medications are expensive; and is not curative.
Following UNAIDS global report on AIDS epidemic 2013 \cite{UNAIDS_report_2013}, globally,
an estimated 35.3 million people were living with HIV in 2012. An increase from previous years,
as more people are receiving ART. There were approximately 2.3 million new HIV infections globally,
showing a 33\% decline in the number of new infections with respect to 2001. At the same time,
the number of AIDS deaths is also declining with around 1.6 million AIDS deaths in 2012,
down from about 2.3 million in 2005.

Mycobacterium tuberculosis is the cause of most occurrences of tuberculosis (TB)
and is usually acquired via airborne infection from someone who has active TB.
It typically affects the lungs (pulmonary TB) but can affect other sites as well
(extrapulmonary TB). In 2012, approximately 8.6 million people fell ill with TB
and 1.3 million people died from TB. Nevertheless, TB death rate dropped 45 per cent
between 1990 and 2012, and 22 million lives were saved through use of strategies
recommended by WHO \cite{TB_WHO_report_2013}.

Individuals infected with HIV are more likely to develop TB disease because of their
immunodeficiency, and HIV infection is the most powerful risk factor for progression
from TB infection to disease \cite{Getahun:etall:CID2010}. In 2012, 1.1 million
of 8.6 million people who developed TB worldwide were HIV-positive. The number
of people dying from HIV-associated to TB has been falling since 2003. However,
there were still 320 000 deaths from HIV-associated to TB in 2012, and further efforts
are needed to reduce this burden \cite{TB_WHO_report_2013}. ART is a critical intervention
for reducing the risk of TB morbidity and mortality among people living with HIV and,
when combined with TB preventive therapy, it can have a significant impact
on TB prevention \cite{TB_WHO_report_2013}.

Collaborative TB/HIV activities (including HIV testing, ART therapy and TB preventive measures)
are crucial for the reduction of TB-HIV coinfected individuals. WHO estimates that these collaborative
activities prevented 1.3 million people from dying, from 2005 to 2012. However, significant
challenges remain: the reduction of tuberculosis related deaths among people living with HIV
has slowed in recent years; the ART therapy is not being delivered to TB-HIV coinfected patients
in the majority of the countries with the largest number of TB/HIV patients; the pace of treatment
scale-up for TB/HIV patients has slowed; less than half of notified TB patients were tested
for HIV in 2012; and only a small fraction of TB/HIV infected individuals received
TB preventive therapy \cite{UNAIDS_report_2013}.

The study of the joint dynamics of TB and HIV present formidable mathematical challenges
due to the fact that the models of transmission are quite distinct \cite{CChavez_TB_HIV_2009}.
Some mathematical models have been proposed for TB-HIV coinfection (see, e.g.,
\cite{Mod:TB:HIV:Bacaer:JMB:2008,Bhunu:BMB:2009:HIV:TB,Kirschner:TB:HIV:1999,%
Mod:TB:HIV:Magombedze:MathPop:2010,Naresh:TB:HIV:2005,CChavez_TB_HIV_2009,Song:TB:HIV:2008}).
In this paper, we propose a new population model for TB-HIV/AIDS coinfection transmission dynamics,
where TB, HIV and TB-HIV infected individuals have access to respective disease treatment,
and single HIV-infected and TB-HIV co-infected individuals under HIV and TB/HIV treatment,
respectively, stay in a \emph{chronic} stage of the HIV infection.

Optimal control is a branch of mathematics developed
to find optimal ways to control a dynamic system
\cite{Cesari_1983,Fleming_Rishel_1975,Pontryagin_et_all_1962}.
While the usefulness of optimal control theory in epidemiology is nowadays well recognized (see, e.g.,
\cite{MariaRPinho:SEIR:MBE:2014,UrszulaLed:OC:SIR:AIMS:2011,livro_Lenhart_2007,%
Rodrigues_Monteiro_Torres_2010,Rodrigues_Monteiro_Torres_Zinober_2011}),
and has been applied to TB models (see, e.g.,
\cite{Bowong_2010,Emvudu_et_all,Haffat_et_all,SLenhart_2002,Silva:Torres:TBOC:NACO:2012,Silva:Torres:TBOC:MBS:2013})
and HIV models (see, e.g., \cite{Kirschner:Lenhart:OC:HIV:JMB:1996,Magombedze:etall:OC:HIV:2009}),
to our knowledge optimal control have never been applied to a TB-HIV/AIDS coinfection model.
In this paper, we apply optimal control theory to our TB-HIV/AIDS model and study optimal strategies
for the minimization of the number of individuals with TB and AIDS active diseases, taking into account
the costs associated to the proposed control measures.

The paper is organized as follows. The model is formulated in Section~\ref{sec:model}.
The HIV-only and TB-only sub-models of the full TB-HIV/AIDS model are analyzed
in Section~\ref{sec:submodels} and the full TB-HIV/AIDS model is analyzed
in Section~\ref{sec:fullmodel}. In Section~\ref{sec:optimal:control}
we propose an optimal control problem and apply the Pontryagin maximum principle
to derive its solution. In Section~\ref{sec:num:results} numerical
simulations and discussion of the results are carried out for the optimal control
problem associated to the TB-HIV/AIDS model. We end mentioning some
possible future work in Section~\ref{sec:FC:fw}.


\section{Model formulation and basic properties}
\label{sec:model}

The model subdivides the human population into eleven mutually-exclusive compartments,
namely susceptible individuals ($S$), TB-latently infected individuals, who have
no symptoms of TB disease and are not infectious ($L_T$), TB-infected individuals,
who have active TB disease and are infectious ($I_T$), TB-recovered individuals
($R$), HIV-infected individuals with no clinical symptoms of AIDS ($I_H$),
HIV-infected individuals under treatment for HIV infection ($C_H$),
HIV-infected individuals with AIDS clinical symptoms ($A$), TB-latent individuals
co-infected with HIV (pre-AIDS) ($L_{TH}$), HIV-infected individuals (pre-AIDS)
co-infected with active TB disease ($I_{TH}$), TB-recovered individuals with
HIV-infection without AIDS symptoms ($R_{H}$), HIV-infected individuals
with AIDS symptoms co-infected with active TB ($A_T$).
The total population at time $t$, denoted by $N(t)$, is given by
\begin{multline*}
N(t) = S(t) + L_T(t) + I_T(t) + R(t) + I_H(t) + C_H(t) + A(t)\\
+ I_{TH}(t) + L_{TH}(t) + R_{H}(t) + A_T(t).
\end{multline*}
The susceptible population is increased by the recruitment of individuals
(assumed susceptible) into the population, at a rate $\Lambda$.
All individuals suffer from natural death, at a constant rate $\mu$.
Susceptible individuals acquire TB infection from individuals
with active TB at a rate $\lambda_T$, given by
\begin{equation}
\label{eq:lambdaT}
\lambda_T(t) = \frac{\beta_1}{N(t)} \left(I_T(t) + I_{TH}(t) + A_T(t)\right),
\end{equation}
where $\beta_1$ is the effective contact rate for TB infection.
Similarly, susceptible individuals acquire HIV infection, following
effective contact with people infected with HIV at a rate $\lambda_H$, given by
\begin{equation}
\label{eq:lambdaH}
\lambda_H(t) = \frac{\beta_2}{N(t)} \left[ I_H(t) + I_{TH}(t) + L_{TH}(t)
+ R_{H}(t) + \eta_C \, C_H(t)  + \eta_A \left(A(t) + A_T(t)\right) \right],
\end{equation}
where $\beta_2$ is the effective contact rate for HIV transmission.
The modification parameter $\eta_A \geq 1$ accounts for the relative
infectiousness of individuals with AIDS symptoms, in comparison to those
infected with HIV with no AIDS symptoms. Individuals with AIDS symptoms
are more infectious than HIV-infected individuals (pre-AIDS) because
they have a higher viral load and there is a positive correlation
between viral load and infectiousness \cite{art:viral:load}.
On the other hand, $\eta_C \leq 1$ translates the partial restoration of immune function
of individuals with HIV infection that use correctly ART \cite{AIDS:chronic:Lancet:2013}.

\begin{remark}
For the basic and classical SIR model, one has the force of infection
$F$ given by $F = \beta I$, which models the transition rate from
the compartment of susceptible individuals $S$
to the compartment of infectious individuals $I$.
However, for large classes of communicable diseases,
it is more realistic to consider a force of infection $F$
that does not depend on the absolute number of infectious,
but on their fraction with respect to the total population $N$,
that is, $F = \beta \frac{I}{N}$. In our case, the force of infection
for the HIV is given by
$\lambda_H = \frac{\beta_2}{N} \left[ I_H + I_{TH} + L_{TH}
+ R_{H} + \eta_C \, C_H  + \eta_A \left(A + A_T\right) \right]$.
\end{remark}

Only approximately 10\% of people infected
with mycobacterium tuberculosis develop active TB disease. Therefore,
approximately 90\% of people infected remain latent. Latent infected TB people
are asymptomatic and do not transmit TB \cite{Styblo_1978}.
Individuals leave the latent-TB class $L_T$ by becoming infectious,
at a rate $k_1$, or recovered, with a treatment rate $\tau_1$.
The treatment rate for active TB-infected individuals is $\tau_2$.
We assume that TB-recovered individuals $R$ acquire partial immunity
and the transmission rate for this class is given by $\beta'_1 \lambda_T$
with $\beta'_1 \leq 1$. Individuals with active TB disease suffer
induced death at a rate $d_T$. We assume that individuals in the class
$R$ are susceptible to HIV infection at a rate $\lambda_H$.
On the other hand, TB-active infected individuals $I_T$ are susceptible
to HIV infection, at a rate $\delta \lambda_H$, where the modification
parameter $\delta \geq 1$ accounts for higher probability of individuals
in class $I_T$ to become HIV-positive.
HIV-infected individuals (with no AIDS symptoms) progress to the AIDS class $A$
at a rate $\rho_1$, and to the class of individuals with HIV infection under
treatment $C_H$ at a rate $\phi$. Individuals in the class $C_H$ leave
to the class $I_H$ at a rate $\omega_1$. HIV-infected individuals
with AIDS symptoms are treated for HIV at the rate $\alpha_1$
and suffer induced death at a rate $d_A$. Individuals in the class
$I_H$ are susceptible to TB infection at a rate $\psi \lambda_T$,
where $\psi \geq 1$ is a modification parameter
traducing the fact that HIV infection is a driver of TB epidemic
\cite{Kwan_Ernst_HIV_TB_Syndemic}.
HIV-infected individuals (pre-AIDS) co-infected with TB-disease,
in the active stage $I_{TH}$, leave this class at a rate $\rho_2$.
A fraction $p$ of $I_{TH}$ individuals take simultaneously TB
and HIV treatment and a fraction $q$ of $I_{TH}$ individuals take only TB treatment.
Individuals in the class $I_{TH}$ progress to the class $C_H$ at a rate $p \rho_2$
and to the class $R_H$ at a rate $q p_2$. Individuals in the class $I_{TH}$
that do not take any  of the TB or HIV treatments progress to the class $A_T$
at a rate $(1-(p+q)) \rho_2$, and suffer TB induced death rate at a rate $d_T$.
Individuals leave $L_{TH}$ class at a rate $\tau_3$. A fraction $r$ of $L_{TH}$
individuals take simultaneously TB and HIV treatment
and a fraction $1-r$ take only TB treatment.
Individuals in the class $L_{TH}$ progress to the class $C_H$ at a rate
$r \tau_3$ and to the class $R_H$ at a rate $(1-r) \tau_3$.
Individuals in the class $L_{TH}$ are more likely to progress
to active TB disease than individuals infected only with latent TB.
In our model, this progression rate is given by $k_2$.
Similarly, HIV infection makes individuals more susceptible to TB reinfection
when compared with non HIV-positive patients. The modification parameter associated
to the TB reinfection rate, for individuals in the class $R_{H}$, is given by $\beta'_2$,
where $\beta'_2 \geq 1$. Individuals in this class progress to class $A$, at a rate $\omega_2$.
HIV-infected individuals (with AIDS symptoms), co-infected with TB, are treated for HIV,
at a rate $\alpha_2$. Individuals in the class $A_T$ suffer from AIDS-TB
coinfection induced death rate, at a rate $d_{TA}$.
The aforementioned assumptions result in the system of differential
equations
\begin{equation}
\label{model:TB:HIV:Chronic:semcont}
\begin{cases}
\dot{S}(t) = \Lambda - \lambda_T(t) S(t) - \lambda_H(t) S(t) - \mu S(t),\\
\dot{L}_T(t) = \lambda_T(t) S(t) + \beta^{'}_1 \lambda_T(t) R(t) - (k_1 + \tau_1 + \mu)L_T(t),\\
\dot{I}_T(t) = k_1 L_T(t) - \left(\tau_2 +d_T +\mu + \delta \lambda_H(t)\right)I_T(t),\\
\dot{R}(t) = \tau_1 L_T(t) + \tau_2 I_T(t) - (\beta^{'}_1 \lambda_T(t) + \lambda_H(t) + \mu) R(t),\\
\dot{I}_H(t) = \lambda_H(t) S(t) - (\rho_1 + \phi + \psi \lambda_T(t) + \mu)I_H(t)
+ \alpha_1 A(t) + \lambda_H(t) R(t) + \omega_1 C_H(t), \\
\dot{A}(t) =  \rho_1 I_H(t) + \omega_2 R_H(t) - \alpha_1 A(t) - (\mu + d_A) A(t),\\
\dot{C}_H(t) = \phi I_H(t) + p \, \rho_2 I_{TH}(t) + r\,  \tau_3 L_{TH}(t) - (\omega_1 + \mu)C_H(t),\\
\dot{L}_{TH}(t) = \beta^{'}_2 \lambda_T(t) R_{H}(t) - (k_2 + \tau_3 + \mu) L_{TH}(t),\\
\dot{I}_{TH}(t) = \delta \lambda_H(t) I_T(t) + \psi \lambda_T(t) I_H(t)
+ \alpha_2 A_T(t)+ k_2 L_{TH}(t) - \left( \rho_2 + \mu + d_T \right)I_{TH}(t),\\
\dot{R}_{H}(t) = q \rho_2 I_{TH}(t) + (1-r)\, \tau_3 L_{TH}(t)
- \left(\beta^{'}_2 \lambda_T(t) + \omega_2 + \mu\right)R_{H}(t),\\
\dot{A}_T(t) = (1-(p+q))\rho_2 I_{TH}(t) -(\alpha_2 + \mu + d_{TA})A_T(t),
\end{cases}
\end{equation}
that describes the transmission dynamics of TB and HIV/AIDS disease.
The model flow is illustrated in Figure~\ref{fig:model:flow}.
\begin{figure}
\centering
\includegraphics[scale=0.85]{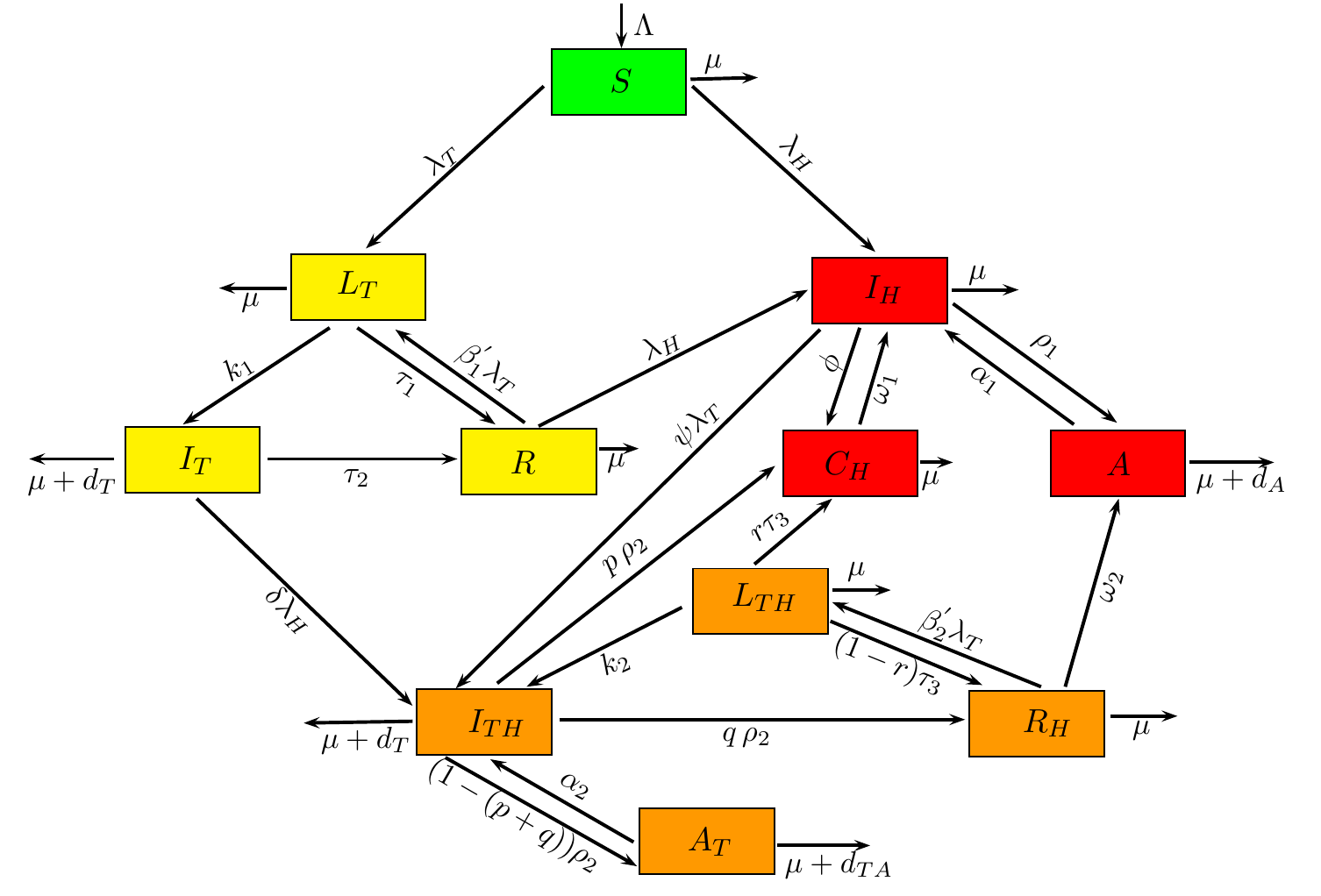}
\caption{Model for TB-HIV/AIDS transmission.}
\label{fig:model:flow}
\end{figure}


\subsection{Positivity and boundedness of solutions}

Since the system of equations \eqref{model:TB:HIV:Chronic:semcont} represents human populations,
all parameters in the model are non-negative and it can be shown that, given non-negative
initial values, the solutions of the system are non-negative. Consider the biologically feasible region
\begin{equation*}
\Omega = \{ \left( S, L_T, I_T, R, I_H, A, C_H, L_{TH}, I_{TH}, R_H, A_T \right)
\in \R_+^{11} \, : \, N \leq \Lambda/\mu \}.
\end{equation*}
In what follows we prove the positive invariance of $\Omega$
(i.e., all solutions in $\Omega$ remain in $\Omega$ for all time).
The rate of change of the total population, obtained by adding all the equations
in model \eqref{model:TB:HIV:Chronic:semcont}, is given by
\begin{equation*}
\frac{dN}{dt} = \Lambda - \mu N(t) - d_T I_T(t) - d_A A(t) - d_T I_{TH}(t) - d_{TA} A_T(t).
\end{equation*}
Using a standard comparison theorem \cite{Lakshmikantham:1989} we can show that
$$
N(t) \leq N(0) e^{-\mu t} + \frac{\Lambda}{\mu} \left( 1- e^{-\mu t} \right).
$$
In particular, $N(t) \leq \frac{\Lambda}{\mu}$ if $N(0) \leq \frac{\Lambda}{\mu}$.
Thus, the region $\Omega$ is positively invariant. Hence, it is sufficient to consider
the dynamics of the flow generated by \eqref{model:TB:HIV:Chronic:semcont} in $\Omega$.
In this region, the model is epidemiologically and mathematically well posed \cite{Hethcote:2000}.
Thus, every solution of the model \eqref{model:TB:HIV:Chronic:semcont} with initial conditions
in $\Omega$ remains in $\Omega$ for all $t > 0$.
This result is summarized below.

\begin{lemma}
\label{lemma:posit:invariant:model:TB:HIV}
The region $\Omega$ is positively invariant for the model
\eqref{model:TB:HIV:Chronic:semcont} with non-negative initial conditions
in $\R^{11}_{+}$.
\end{lemma}


\section{Analysis of the sub-models}
\label{sec:submodels}

In this section we analyze the models for HIV only (HIV-\emph{only model})
and TB only (TB-\emph{only model}).


\subsection{HIV-only model}

The model that considers only HIV (obtained by setting
$L_T = I_T = R = L_{TH}= I_{TH}= R_H= A_T$) is given by
\begin{equation}
\label{model:HIV:only}
\begin{cases}
\dot{S}(t) = \Lambda - \lambda_H(t) S(t) - \mu S(t),\\[0.2 cm]
\dot{I}_H(t) = \lambda_H(t) S(t) - (\rho_1 + \phi + \mu)I_H(t)
 + \alpha_1 A(t)  + \omega_1 C_H, \\[0.2 cm]
\dot{A}(t) =  \rho_1 I_H(t) - (\alpha_1 + \mu + d_A) A(t),\\[0.2 cm]
\dot{C}_H(t) = \phi I_H(t) - (\omega_1 + \mu)C_H(t),
\end{cases}
\end{equation}
where
\begin{equation*}
\lambda_H(t) = \frac{\beta_2}{N(t)} \left[ I_H(t) + \eta_C C_H(t) + \eta_A A(t) \right]
\end{equation*}
with
\begin{equation*}
N(t) = S(t) + I_H(t) + A(t) + C_H(t).
\end{equation*}
Analogously to Lemma~\ref{lemma:posit:invariant:model:TB:HIV}, we can prove that the region
\begin{equation}
\label{Omega1:inv:region:HIV}
\Omega_1 = \{ \left( S, I_H, A, C_H \right) \in \R_+^{4} \, : \, N \leq \Lambda/\mu \}
\end{equation}
is positively invariant and attracting. Thus, the dynamics
of the HIV-only model will be considered in $\Omega_1$.


\subsubsection{Persistence}

In this section, we look for the conditions under which the host population
and disease will persist. Rewriting the submodel system \eqref{model:HIV:only} as
\begin{equation}
\label{mod:HIV:beta2N}
\begin{cases}
\dot{S}(t) = \Lambda - \frac{\beta_2(N) \left( I_H + \eta_C \, C_H
+ \eta_A A  \right)}{N} S(t) - \mu S(t),\\[0.2 cm]
\dot{I}_H(t) = \frac{\beta_2(N) \left( I_H + \eta_C \, C_H
+ \eta_A A  \right)}{N} S(t) - (\rho_1 + \phi + \mu)I_H(t)
 + \alpha_1 A(t)  + \omega_1 C_H, \\[0.2 cm]
\dot{C}_H(t) = \phi I_H(t) - (\omega_1 + \mu)C_H(t),\\[0.2 cm]
\dot{A}(t) =  \rho_1 I_H(t) - (\alpha_1 + \mu + d_A) A(t),
\end{cases}
\end{equation}
in what follows we assume that $\beta_2 (N)$ is continuous for $N \geq 0$
and continuously differentiable for $N > 0$; $\beta_2(N)$
is monotone nondecreasing in $N$; and $\beta_2(N) > 0$ if $N > 0$.

\begin{remark}
In this work $\beta_2$ denotes the effective contact rate for HIV transmission.
It is a constant for a concrete situation, but one can look to it
as variable in the sense that, depending on the situation/region,
one can have different values for this parameter. This is so because
$\beta_2$ is related with the level of contagion/propagation of the disease.
In Section~\ref{sec:optimal:control} we consider fixed values for $\beta_1$ and $\beta_2$,
which represent specific cases of the infection level.
By varying $\beta_1$ and $\beta_2$ we vary
the basic reproduction numbers (see expressions
\eqref{eq:R0:HIV:only} and \eqref{eq:brn:R2}
for $R_1$ and $R_2$, respectively).
Here we consider $\beta_2$ as a function of $N$
to discuss persistence.
\end{remark}

It is convenient to reformulate the model in terms of the fractions
of the $S$, $I_H$, $A$ and $C_H$ parts of the population,
\begin{equation}
\label{eq:xyzw}
x = \frac{S}{N}, \quad y = \frac{I_H}{N},
\quad z = \frac{C_H}{N}, \quad w = \frac{A}{N},
\end{equation}
and express \eqref{mod:HIV:beta2N} in these terms to obtain the system
\begin{equation}
\label{system:xyzw}
\begin{cases}
\dot{N} =  \Lambda - (\mu + d_A w) N ,\\[0.2 cm]
\dot{x} = \frac{\Lambda}{N}(1-x) - \beta_2(N) (y + \eta_C z
+ \eta_A w)x + x(\mu - y + d_A w) ,\\[0.2 cm]
\dot{y} =  \beta_2(N)(y+ \eta_C Z + \eta_A w)x + y d_A w
- \left(\rho_1 + \phi + \mu\right)y + \alpha_1 w + \omega_1 z,\\[0.2 cm]
\dot{z} = \phi y - (\omega_1 + \mu - d_A w)z,\\[0.2 cm]
\dot{w} = \rho_1 y - (\alpha_1 + \mu + d_A - d_A w) w.
\end{cases}
\end{equation}
Equations \eqref{eq:xyzw} suggest that $x + y + z + w = 1$.
The manifold $x + y + z + w= 1$, $x, y, z, w \ge 0$, is forward invariant
under the solution flow of \eqref{system:xyzw}, which has a global solution satisfying
\eqref{eq:xyzw}. We now show conditions under which the host population will persist.

\begin{theorem}
\label{theo:persistence}
Let $\beta_2(0) = 0$, $N(0) > 0$. Then the population
is uniformly persistent, that is,
$$
\liminf_{t \to \infty} N(t) \geq \varepsilon,
$$
where $\varepsilon > 0$ does not depend on the initial data.
\end{theorem}

\begin{proof}
We have to show that the set
\begin{equation*}
X_2 = \left\{ N = 0, \, x \geq 0 , \, y \geq 0 , \, z \geq 0, \,
w \geq 0 , \, x + y + z + w = 1 \right\}
\end{equation*}
is uniform strong repeller for
\begin{equation*}
X_1 = \left\{ N > 0, \, x \geq 0 , \, y \geq 0 , \, z \geq 0,
\, w \geq 0 , \, x + y + z + w = 1  \right\}.
\end{equation*}
Theorem~\ref{theo:auxiliar1}, Theorem~\ref{theo:auxiliar2} and Corollary~\ref{cor:auxiliar3}
are taken from \cite{Bhunu:BMB:2009:HIV:TB,Thieme:1993}.
\begin{theorem}
\label{theo:auxiliar1}
Let $X$ be a locally compact metric space with metric $d$. Let $X$ be the disjoint
union of two sets $X_1$ and $X_2$ such that $X_2$ is compact. Let $\Phi$ be a continuous
semiflow on $X_1$. Then $X_2$ is a uniform strong repeller for $X_1$,
whenever it is a uniform weak repeller for $X_1$.
\end{theorem}
\begin{theorem}
\label{theo:auxiliar2}
Let $D$ be a bounded interval in $\R$ and $g \, : \, (t_0, \infty) \times D \to \R$
be bounded and uniformly continuous. Further, let $x \, : \, (t_0, \infty) \to D$ be a solution of
$$
x' = g(t, x),
$$
which is defined on the whole interval $(t_0, \infty)$.
Then there exist sequences $s_n, t_n \to \infty$ such that
$$
\lim_{n \to \infty} g(s_n, x_\infty) = 0
= \lim_{n \to \infty} g(t_n, x^\infty).
$$
\end{theorem}
\begin{corollary}
\label{cor:auxiliar3}
Let the assumptions of Theorem~\ref{theo:auxiliar2} be satisfied. Then
\begin{itemize}
\item[a)] $\liminf_{t \to \infty} g(t, x_\infty) \geq 0 \geq \limsup_{t \to \infty} g(t, x_\infty)$,
\item[b)] $\liminf_{t \to \infty} g(t, x^\infty) \geq 0 \geq \limsup_{t \to \infty} g(t, x^\infty)$.
\end{itemize}
\end{corollary}
As the assumptions of Theorem~\ref{theo:auxiliar1} are satisfied, it suffices to show
that $X_2$ is a uniform weak repeller for $X_1$. Let $r = y + z + w$. Then,
\begin{equation*}
\begin{split}
r' &= \beta_2(N)(yx + \eta_C zx + \eta_A wx) + d_A w r - \mu r - d_A w \\
&\leq \beta_2(N)(1 + \eta_C + \eta_A) - \frac{\Lambda}{N} r + d_A (r-1),
\end{split}
\end{equation*}
using the fact that $x, y, z, w, r \leq 1$. This implies that
\begin{equation}
\label{eq:betaN:infty}
\begin{split}
&\frac{\Lambda}{N^\infty} r^\infty + (1-r^\infty)d_A
\leq \beta_2(N^\infty)(1 + \eta_C + \eta_A)\\
&  \Rightarrow \beta_2(N^\infty) \geq \frac{\Lambda r^\infty}{N^\infty
\left(1 + \eta_C + \eta_A\right)}
+ \frac{(1-r^\infty)d_A}{1 + \eta_C + \eta_A}.
\end{split}
\end{equation}
From the equation of $N$ in \eqref{system:xyzw} we have
\begin{equation*}
\liminf_{t \to \infty} \frac{1}{N} \frac{dN}{dt} \geq \frac{\Lambda}{N^\infty}
- \left(\mu + d_A w^\infty \right) \geq \frac{\Lambda}{N^\infty}
- \left(\mu + d_A r^\infty \right).
\end{equation*}
Hence $N$ increases exponentially, unless
\begin{equation}
\label{eq:lambda:Ninfty}
\frac{\Lambda}{N^\infty} \leq \mu + d_A r^\infty,
\quad \text{that is,} \quad
\frac{1}{d_A} \left( \frac{\Lambda}{N^\infty} - \mu \right) \leq r^\infty.
\end{equation}
Combining \eqref{eq:betaN:infty} and \eqref{eq:lambda:Ninfty}, we obtain that
\begin{equation}
\label{eq:betaN:leq}
\beta_2(N^\infty) \geq \left(\frac{\Lambda}{d_A N^\infty (1 + \eta_C + \eta_A)}
- \frac{1}{1+\eta_C + \eta_A} \right) \left(\frac{\Lambda}{N^\infty} - \mu \right)
+ \frac{d_A}{1+\eta_C + \eta_A},
\end{equation}
as $\beta_2(0) = 0$ and $\beta_2(N)$ is continuous at $0$, $N^\infty \geq \varepsilon > 0$
with $\varepsilon$ not depending on the initial data. From \eqref{eq:betaN:leq}
we see that we can relax $\beta_2(0) = 0$ and require
\begin{equation*}
\beta_2(0) < \left(\frac{\Lambda}{d_A N^\infty (1 + \eta_C + \eta_A)}
- \frac{1}{1+\eta_C + \eta_A} \right) \left(\frac{\Lambda}{N^\infty}
- \mu \right) + \frac{d_A}{1+\eta_C + \eta_A}.
\end{equation*}
This concludes the proof.
\end{proof}

The disease is persistent in the population if the fraction of the infected
and AIDS cases is bounded away from zero. If the population dies out and the
fraction of the infected and AIDS remains bounded away from zero, we would
still say that the disease is persistent in the population.

\begin{proposition}
\label{prop:persistent}
Let $\beta_2(\infty) (1 + \eta_C + \eta_A) \geq \frac{\Lambda}{N^\infty} r^\infty$.
Then the disease is uniformly weakly persistent insofar as
\begin{equation*}
r^\infty = \limsup_{t \to \infty} r(t) \geq \varepsilon,
\end{equation*}
with $\varepsilon > 0$ being independent of the initial data, provided that $r(0) > 0$.
\end{proposition}

The proof of Proposition~\ref{prop:persistent} is outlined in \cite{Thieme:1993}.


\subsubsection{Local stability of disease-free equilibrium}
\label{sec:DFE:HIVmodel}

The model \eqref{model:HIV:only} has a disease-free equilibrium (DFE),
obtained by setting the right-hand sides of the equations
in the model to zero, given by
\begin{equation*}
\Sigma_0 = \left(S^*, I^*_H, A^*, C^*_H  \right)
= \left(\frac{\Lambda}{\mu},0, 0,0  \right).
\end{equation*}
The linear stability of $\Sigma_0$ can be established using the
next-generation operator method on the system \eqref{model:HIV:only}.
Following \cite{van:den:Driessche:2002}, the basic reproduction number is obtained
as the spectral radius of the matrix $F  V^{-1}$ at the DFE, $\Sigma_0$, with
$F$ and $V$ given by, respectively,
\begin{equation*}
F =  \left[ \begin {array}{cccc} 0&0&0&0\\ \noalign{\medskip}
\lambda_H &{\frac {{\beta_2}\,S}{{N}}}&{\frac {{\beta_2}\,{
\eta_A}\,S}{{N}}}&{\frac {{\beta_2}\,{\eta_C}\,S}{{N}}
}\\ \noalign{\medskip}
0&0&0&0\\ \noalign{\medskip}
0&0&0&0\end {array}
 \right]
\end{equation*}
and
\begin{equation*}
V = \left[ \begin {array}{cccc}
\lambda_H +\mu&{\frac {{
\beta_2}\,S}{{N}}}&{\frac {{\beta_2}\,{\eta_A}\,S}{{N}
}}&{\frac {{\beta_2}\,{\eta_C}\,S}{{N}}}\\ \noalign{\medskip}
0&C_1&-{\alpha_1}&-{\omega_1}\\ \noalign{\medskip}
0&-{\rho_1}&C_2&0\\ \noalign{\medskip}
0&-\phi&0&C_3
\end{array} \right],
\end{equation*}
where $C_1 = \rho_1 + \phi + \mu$, $C_2 = \alpha_1 + \mu + d_A$, $C_3 = \omega_1 + \mu$.
The basic reproduction number is given by the dominant eigenvalue
of the matrix $F  V^{-1}$, that is,
\begin{equation}
\label{eq:R0:HIV:only}
R_1 = \frac{{\beta_2}\,\Lambda\, \left( C_3(C_2 +{\eta_A}\,{\rho_1})
+{\eta_C}\,\phi C_2 \right)}{N \mu \left[ \mu \left(C_3(\rho_1 + C_2)
+ C_2 \phi + \rho_1 d_A \right) + \rho_1 \omega_1 d_A  \right]}.
\end{equation}
The basic reproduction number $R_1$ represents the expected average number of new HIV infections
produced by a single HIV-infected individual when in contact with a completely susceptible
population \cite{van:den:Driessche:2002}.

\begin{remark}
\label{rem3}
The next-generation matrix is one of the most well known methods in epidemiology
to compute the basic reproduction number for a compartmental model
of the spread of infectious diseases. To calculate the basic reproduction number
through this method, the whole population is divided into $n$ compartments
in which there are $m<n$ infected compartments. Let $x_i$, $i=1,2,\ldots,m$,
be the numbers of infected individuals in the $i$th infected compartment at time $t$.
Now, the epidemic model is $x_i' = F_i(x) - V_i(x)$ or, in vector form,
$x' = F(x) - V(x)$. Let $x_0$ denote here the disease-free equilibrium state.
The Jacobian matrices of $F(x)$ and $V(x)$ are, respectively,
$$
DF(x) = \left[
\begin{array}{cc}
F & 0 \\
0 & 0 \\
\end{array}
\right]
\quad \text{ and } \quad
DV(x) = \left[
\begin{array}{cc}
V & 0 \\
J_3 & J_4 \\
\end{array}
\right],
$$
where $F$ and $V$ are the $m \times m$ matrices given by
$$
F = \left[\frac{\partial F_i(x_0)}{\partial x_j}\right],
\quad
V = \left[\frac{\partial V_i(x_0)}{\partial x_j}\right].
$$
The matrix $F V^{-1}$ is known as the next-generation matrix
and its spectral radius is the basic reproduction number of the model.
The reader interested in all the details about the computation of the basic reproduction
number by the next-generation matrix is referred to \cite{MR1057044,van:den:Driessche:2002}
or any good book on dynamical modeling and analysis of epidemics (e.g., \cite{MR1882991}).
\end{remark}

\begin{lemma}
The disease free equilibrium $\Sigma_0$ is locally asymptotically stable
if $R_1 < 1$, and unstable if $R_1 > 1$.
\end{lemma}

\begin{proof}
Following Theorem~2 of \cite{van:den:Driessche:2002}, the disease-free equilibrium,
$\Sigma_0$, is locally asymptotically stable if all the eigenvalues of the Jacobian
matrix of the system \eqref{model:HIV:only}, here denoted by $M\left(\Sigma_0\right)$,
computed at the DFE $\Sigma_0$, have negative real parts.
The Jacobian matrix of the system \eqref{model:HIV:only}
at disease free equilibrium $\Sigma_0$ is given by
\begin{equation}
\label{eq:Jacob:DFE}
M \left( \Sigma_0 \right)
= \left[ \begin {array}{cccc}
-\mu&-{\frac {\beta_2\,\Lambda}{\mu\,N}}
&-{\frac {\beta_2\,\eta_A\,\Lambda}{\mu\,N}}&-{\frac {\beta_2\,\eta_C\,
\Lambda}{\mu\,N}}\\ \noalign{\medskip}
0&{\frac {\beta_2\,\Lambda}{\mu
\,N}}-C_1&{\frac {\beta_2\,\eta_A\,\Lambda}{\mu\,N}}+
\alpha_1&{\frac {\beta_2\,\eta_C\,\Lambda}{\mu\,N}}+\omega_1
\\ \noalign{\medskip}
0&\rho_1&C_2&0\\ \noalign{\medskip}
0&\phi&0&C_3
\end {array} \right].
\end{equation}
One has
\begin{equation*}
\textrm{trace}\left[ M\left(\Sigma_0\right) \right] =
-\mu+{\frac {\beta_2\,\Lambda}{\mu\,N}}-(C_1 + C_2 + C_3) < 0
\end{equation*}
and
\begin{multline*}
\det \left[ M\left(\Sigma_0\right) \right] = -\frac{\beta_2\,\Lambda}{N}
\left[ C_3(C_2 + \rho_1 \eta_A) + \phi \eta_C C_2 \right]\\
+ \mu \left[ \mu \left(C_3(\rho_1 + C_2) + C_2 \phi + \rho_1 d_A \right)
+ \rho_1 \omega_1 d_A \right] > 0
\end{multline*}
for $R_1  < 1$.  We have just proved that the disease free equilibrium $\Sigma_0$ of model
\eqref{model:HIV:only} is locally asymptotically stable if $R_1 < 1$,
and unstable if $R_1 > 1$.
\end{proof}


\subsubsection{Global stability of disease-free equilibrium (DFE)}

Following \cite{CChavez:Feng:Huang:2002}, let us rewrite
the submodel system \eqref{model:HIV:only} as
\begin{equation}
\label{mod:global:stab:DFE:HIV}
\begin{split}
&\frac{dX}{dt} = F(X, Z), \\
&\frac{dZ}{dt} = G(X, Z), \quad G(X, 0) = 0,
\end{split}
\end{equation}
where $X = S$ and $Z = (I_H, A, C_H)$,
with $X \in \R_+$ denoting the total number of uninfected individuals
and $Z \in \R^3_+$ denoting the total number of infected individuals.
The disease-free equilibrium is now denoted by
\begin{equation*}
U_0 = (X_0, 0), \quad \text{where} \, \, X_0
= \left(\frac{\Lambda}{\mu}, 0  \right).
\end{equation*}
The conditions (H1) and (H2) below must be met
to guarantee global asymptotically stability:
\begin{itemize}
\item[(H1)] for $\frac{dX}{dt} = F(X, 0)$, $U_0$ is globally asymptotically stable;
\item[(H2)] $G(X, Z) = AZ - \hat{G}(X, Z)$, $\hat{G}(X, Z) \geq 0$ for
$(X, Z) \in \mathcal{G}$, where $A = D_Z G(U_0, 0)$ is a Metzler matrix
(the off diagonal elements of $A$ are nonnegative) and $\mathcal{G}$
is the region where the model makes biological sense.
\end{itemize}

\begin{theorem}
The fixed point $U_0 = (X_0, 0)$ is a globally asymptotically stable equilibrium
of \eqref{model:HIV:only} provided $R_1 < 1$
and the assumptions (H1) and (H2) are satisfied.
\end{theorem}

\begin{proof}
We have
\begin{equation*}
\frac{dX}{dt} = F(X, Z)
= \left[
\begin{array}{c}
\Lambda - \lambda_H S - \mu S
\end{array} \right],
\end{equation*}
\begin{equation*}
F(X, 0) = \left[ \begin{array}{c}
\Lambda - \mu S
\end{array} \right],
\end{equation*}
\begin{equation*}
\frac{dZ}{dt} = G(X, Z)
 = \left[
\begin{array}{c}
\lambda_H(t) S(t) - C_1 I_H(t) + \alpha_1 A(t)  + \omega_1 C_H \\[0.2 cm]
\rho_1 I_H(t) - C_2 A(t)\\[0.2 cm]
\phi I_H(t) - C_3 C_H(t)
\end{array}
\right],
\end{equation*}
and $G(X, 0) = 0$. Therefore,
\begin{equation*}
\frac{dX}{dt} = F(X, 0) = \left[
\begin{array}{c}
\Lambda - \mu S \\[0.2 cm]
- \mu R_T
\end{array} \right],
\end{equation*}
\begin{equation*}
A = D_Z G(X_0, 0) =
\left[ \begin{array}{ccc}
{\frac {\beta_2 \,\Lambda}{N\mu}}-C_1&{\frac {\beta_2\,\eta_A\,\Lambda}{N\mu}}
+\alpha_1&{\frac{\beta_2 \,\eta_C \,\Lambda}{N\mu}}+\omega_1\\
\noalign{\medskip} \rho_1 &-C_2&0\\
\noalign{\medskip} \phi&0&-C_3
\end{array} \right]
\end{equation*}
and
\begin{equation}
\label{eq:hat:G}
\hat{G}(X, Z)
= \left[
\begin{array}{c}
\hat{G}_1(X, Z)\\[0.2 cm]
\hat{G}_2(X, Z)\\[0.2 cm]
\hat{G}_3(X, Z)
\end{array}
\right]
= {\footnotesize{  \left[ \begin {array}{c}
\beta_2(1-\frac{1}{N})(I_H + \eta_A A + \eta_C C_H)\\
\noalign{\medskip}
0\\ \noalign{\medskip}
0
\end {array} \right] }}.
\end{equation}
It follows that $\hat{G}_1(X, Z) \geq 0$, $\hat{G}_2(X, Z) = \hat{G}_3(X, Z) = 0$.
Thus, $\hat{G}(X, Z) \geq 0$. Conditions (H1) and (H2) are satisfied,
and we conclude that $U_0$ is globally asymptotically stable for $R_1 < 1$.
\end{proof}


\subsubsection{Existence of an endemic equilibrium}

To find conditions for the existence of an equilibrium for which HIV is endemic
in the population (i.e., at least one of $I_H^*$, $A^*$ or $C_H^*$ is non-zero),
denoted by $\Sigma_H = \left(S^*, I_H^*, A^*, C_H^* \right)$, the equations
in \eqref{model:HIV:only} are solved in terms of the force of infection
at steady-state ($\lambda_H^*$), given by
\begin{equation}
\label{eq:lambdaH:ast}
\lambda_H^* = \frac{\beta_2 \left( I_H^* + \eta_A A^* + \eta_C \, C_H^* \right)}{N^*}.
\end{equation}
Setting the right hand sides of the model to zero (and noting that
$\lambda_H = \lambda_H^*$ at equilibrium) gives
\begin{equation}
\label{eq:end:equil:HIV:only}
S^* = \frac{\Lambda}{\lambda_H^* + \mu}, \quad I_H^*
=-\frac{\lambda_H^* \Lambda C_2 C_3}{D}, \quad A^*
= -\frac{\rho_1 \lambda_H^* \Lambda C_3}{D}, \quad C_H^*
=- \frac{\phi \lambda_H^* \Lambda C_2}{D},
\end{equation}
with
$D = -(\lambda_H^* + \mu)(\mu \left(C_3(\rho_1 + C_2)
+ C_2 \phi + \rho_1 d_A \right) + \rho_1 \omega_1 d_A)$.
Using \eqref{eq:end:equil:HIV:only} in the expression
for $\lambda_H^*$ in \eqref{eq:lambdaH:ast} shows that
the nonzero (endemic) equilibria of the model satisfy
\begin{equation*}
\lambda_H^* = \frac{\Lambda \beta_2 \left(C_2 C_3 + \eta_A \rho_1 C_3
+ \eta_C \phi C_2\right)}{N \left[ \mu \left(C_3(\rho_1 + C_2)
+ C_2 \phi + \rho_1 d_A \right) + \rho_1 \omega_1 d_A  \right]}-\mu,
\end{equation*}
that is,
\begin{equation*}
\lambda_H^* = -\mu (1-R_1).
\end{equation*}
The force of infection at the steady-state $\lambda_H^*$ is positive,
only if $R_1 > 1$. We have just proved the following result.

\begin{lemma}
The submodel system \eqref{model:HIV:only} has
a unique endemic equilibrium whenever $R_1 > 1$.
\end{lemma}


\subsubsection{Local stability of the endemic equilibrium}

In what follows we prove the local asymptotic stability of the endemic equilibrium
$\Sigma_H$, using the center manifold theory \cite{Carr:1981},
as described in \cite[Theorem~4.1]{CChavez_Song_2004}, with
$\Sigma_H = (S^*, I_H^*, A^*, C_H^*)$ and each of its components given as in \eqref{eq:end:equil:HIV:only}.
To apply this method, the following simplification and change of variables are first made.
Let $S = x_1$, $I_H = x_2$, $A = x_3$ and $C_H = x_4$, so that $N = x_1 + x_2 +x_3 + x_4$.
Further, by using vector notation $X = (x_1, x_2, x_3, x_4)^T$, the submodel \eqref{model:HIV:only}
can be written in the form $\frac{dX}{dt} = (f_1, f_2, f_3, f_4)^T$, as follows:
\begin{equation}
\label{eq:syst:xyzw:HIV}
\begin{split}
\frac{dx_1}{dt} &= f_1 = \Lambda - \frac{\beta_2 \left(x_2 + \eta_A x_3
+ \eta_C x_4  \right)}{x_1 + x_2 +x_3 + x_4} x_1 - \mu x_1,\\[0.2 cm]
\frac{dx_2}{dt} &= f_2 = \frac{\beta_2
\left(x_2 + \eta_A x_3 + \eta_C x_4  \right)}{x_1 + x_2 +x_3 + x_4}
x_1 - C_1 x_2 + \alpha_1 x_3  + \omega_1 x_4,\\[0.2 cm]
\frac{dx_3}{dt} &= f_3 = \rho_1 x_2 - C_2 x_3, \\[0.2 cm]
\frac{dx_4}{dt} &= f_4 = \phi x_2 - C_3 x_4.
\end{split}
\end{equation}
The basic reproduction number of the submodel \eqref{model:HIV:only}
is given by \eqref{eq:R0:HIV:only}. Choose as bifurcation parameter
$\beta^*$, by solving for $\beta_2$ from $R_1 = 1$:
\begin{equation*}
\beta^* = \frac{ \mu \left(C_3(\rho_1 + C_2) + C_2 \phi + \rho_1 d_A \right)
+ \rho_1 \omega_1 d_A }{ C_3(C_2 +{\eta_A}\,{\rho_1}) +{\eta_C}\,\phi C_2  }.
\end{equation*}
The submodel \eqref{model:HIV:only} has a disease free equilibrium given by
\begin{equation*}
\Sigma_{H} = (x_{10}, x_{20}, x_{30}, x_{40}) = \left(\frac{\Lambda}{\mu},0, 0,0  \right).
\end{equation*}
The Jacobian of the system \eqref{eq:syst:xyzw:HIV}, evaluated at
$\Sigma_{0}$, $M\left( \Sigma_0 \right)$, and with
$\beta_2 = \beta^*$, is given by \eqref{eq:Jacob:DFE}.
Note that the above linearized system, of the transformed system \eqref{eq:syst:xyzw:HIV}
with $\beta_2 = \beta^*$, has a zero eigenvalue which is simple.
Hence, the center manifold theory \cite{Carr:1981} can be used to analyze the dynamics
of \eqref{eq:syst:xyzw:HIV} near $\beta_2 = \beta^*$. In particular,
Theorem 4.1 in \cite{CChavez_Song_2004} is used to show the locally
asymptotically stability of the endemic equilibrium point
of \eqref{eq:syst:xyzw:HIV}, for $\beta_2$ near $\beta^*$.

The Jacobian $M(\Sigma_0)$ at $\beta_2 = \beta^*$ has a right eigenvector
(associated with the zero eigenvalue) given by $w = [w_1, w_2, w_3, w_4]^T$, where
\begin{equation*}
\begin{split}
w_1&=-{\frac{w_3 \, \mu(C_3(\rho_1+C_2)+C_2\phi+\rho_1 d_A)
+\rho_1 \omega_1 d_A }{C_3 \mu\,\rho_1}},\\
w_2&= {\frac{w_3\, C_2 }{\rho_1}},\\
w_3&= w_3 > 0,\\
w_4&= {\frac {C_2  \phi\, w_3}{C_3 \rho_1 }}.
\end{split}
\end{equation*}
Further, $M(\Sigma_{0})$ for $\beta_2 = \beta^*$ has a left eigenvector
$v = [v_1, v_2, v_3, v_4]$ (associated with the zero eigenvalue), where
\begin{equation*}
\begin{split}
v_1&=0,\\
v_2&= v_3 {\frac { C_2 C_3 +\eta_A \rho_1 C_3 +\eta_C \phi C_2}
{\mu (\eta_A C_1+ \eta_A  \omega_1+ \alpha_1)+\omega_1 \eta_A \rho_1
+ (\eta_C \phi + \omega_1)\alpha_1}},\\
v_3&= v_3 > 0, \\
v_4&= v_3 {\frac { \eta_C C_2
+ \mu (\omega_1 +\eta_C  \rho_1 +\eta_ C \phi) + \rho_1(\omega_1 \eta_A
+\eta_C d_A)  + (\omega_1 +\eta_C \phi)(\alpha_1 + d_A)}{\mu
\left(\eta_A C_1 +\eta_A  \omega_1 +  \alpha_1\right)
+\omega_1 (\eta_A \rho_1 + \alpha_1) +\eta_C \phi \alpha_1}}.
\end{split}
\end{equation*}
To apply Theorem~4.1 in \cite{CChavez_Song_2004} it is convenient to let
$f_k$ represent the right-hand side of the $k$th equation of the system
\eqref{eq:syst:xyzw:HIV} and let $x_k$ be the state variable whose derivative
is given by the $k$th equation for $k = 1, \ldots, 4$. The local stability
near the bifurcation point $\beta_2 = \beta^*$ is then determined by the
signs of two associated constants, denoted by $a$ and $b$, defined (respectively) by
\begin{equation*}
a = \sum_{k, i, j=1}^4 \, v_k w_i w_j \frac{\partial^2 f_k}{\partial x_i \partial x_j}(0, 0)
\quad \text{and} \quad b = \sum_{k, i =1}^4 \, v_k w_i
\frac{\partial^2 f_k}{\partial x_i \partial \phi}(0, 0)
\end{equation*}
with $\phi = \beta_2 - \beta^*$. Note that, in $f_k(0, 0)$, the first zero corresponds to the DFE,
$\Sigma_0$, for the subsystem \eqref{model:HIV:only}. In other words, $f_k(0, \phi) = 0$,
for $k=1, \ldots, 4$, if and only if the right-hand sides of the equations of \eqref{model:HIV:only}
is zero at $\Sigma_0$. Moreover, from $\phi = \beta_2 - \beta^*$ we have $\phi = 0$
when $\beta_2 = \beta^*$, which is the second component in $f_k(0, 0)$.

For the system \eqref{eq:syst:xyzw:HIV}, the associated non-zero partial
derivatives at the disease free equilibrium $\Sigma_0$ are given by
\begin{equation*}
\begin{split}
&\frac{\partial^2 f_1}{\partial x_2^2} = \frac{2 \beta^* \mu}{\Lambda}, \quad
\frac{\partial^2 f_1}{\partial x_2 \partial x_3}= \frac{\beta^* \mu (1+\eta_A)}{\Lambda}, \quad
\frac{\partial^2 f_1}{\partial x_2 \partial x_4} = \frac{\beta^* \mu (1+\eta_C)}{\Lambda},\\
&\frac{\partial^2 f_1}{\partial x_3^2} = \frac{2 \beta^* \mu \eta_A}{\Lambda}, \quad
\frac{\partial^2 f_1}{\partial x_3 \partial x_4} = \frac{\beta^* \mu (\eta_A+\eta_C)}{\Lambda}, \\
&\frac{\partial^2 f_1}{\partial x_4 \partial x_2} = \frac{\beta^* \mu (1+\eta_C)}{\Lambda}, \quad
\frac{\partial^2 f_1}{\partial x_4^2} = \frac{2 \beta^* \mu \eta_C }{\Lambda}, \\
&\frac{\partial^2 f_2}{\partial x_2^2} = \frac{2 \beta^* \mu}{\Lambda}, \quad
\frac{\partial^2 f_2}{\partial x_2 \partial x_3}= -\frac{\beta^* \mu (1+\eta_A)}{\Lambda}, \quad
\frac{\partial^2 f_2}{\partial x_2 \partial x_4} = -\frac{\beta^* \mu (1+\eta_C)}{\Lambda}, \\
&\frac{\partial^2 f_2}{\partial x_3^2} = -\frac{2 \beta^* \eta_A \mu}{\Lambda}, \quad
\frac{\partial^2 f_2}{\partial x_3 \partial x_4}= -\frac{\beta^* \mu (\eta_A+\eta_C)}{\Lambda},\\
&\frac{\partial^2 f_2}{\partial x_4^2} = -\frac{2 \beta^* \eta_C \mu}{\Lambda}, \quad
\frac{\partial^2 f_2}{\partial x_4 \partial x_2}= -\frac{\beta^* \mu (1+\eta_C)}{\Lambda}.
\end{split}
\end{equation*}
It follows from the above expressions that
\begin{multline*}
a= \frac{w_3^2 C_2 \phi\, \beta_2 \mu \left(\eta_A + \eta_C \right)}{\rho_1 \left(\omega_1+\mu \right)\Lambda}
-\frac{2 D_1 C_2^2 \left(C_3 + \phi(1 + \eta_C)\right)}{D_2 \rho_1^2 C_3}\\
-\frac{2 D_1 \left(C_2 C_3 \rho_1 \left(C_3(1+\eta_A) + \phi\right)
+ C_3^2 \eta_A \rho_1^2 + C_2^2 \phi^2 \eta_C \right)}{D_2 \rho_1^2 C_3^2},
\end{multline*}
with
\begin{equation*}
\begin{split}
D_1 &= v_3 w_3^{2} \beta_2 \mu \left( C_3 C_2 + C_3 \eta_A \rho_1 + \eta_C \phi C_2 \right),\\
D_2 &= \Lambda \left( \eta_A \rho_1 C_3 + C_3 \alpha_1
+ \eta_A \mu C_3 + \mu \phi + \eta_C \phi \alpha_1 \right).
\end{split}
\end{equation*}
For the sign of $b$, it can be shown that
the associated non-vanishing partial derivatives are
\begin{equation*}
\begin{split}
\frac{\partial^2 f_1}{\partial x_2 \partial \beta^*} &= -1,  \quad
\frac{\partial^2 f_1}{\partial x_3 \partial \beta^*} = -\eta_A, \quad
\frac{\partial^2 f_1}{\partial x_4 \partial \beta^*} = -\eta_C, \\
\frac{\partial^2 f_2}{\partial x_2 \partial \beta^*} &= -1, \quad
\frac{\partial^2 f_2}{\partial x_3 \partial \beta^*} = -\eta_A, \quad
\frac{\partial^2 f_2}{\partial x_4 \partial \beta^*} = \eta_C.
\end{split}
\end{equation*}
It also follows from the above expressions that
\begin{equation*}
b = {\frac { ((\eta_A + \eta_C) \rho_1 + C_2) \left( C_3 C_2 +C_3 \eta_A \rho_1
+ \eta_C \phi C_2 \right) v_3 w_3 }{ \left( (\eta_A \rho_1 +\alpha_1+ \eta_A \mu) C_3
+ \eta_A \mu \phi + \eta_C \phi \alpha_1 \right) \rho_1}}.
\end{equation*}
From the previous computations, we have $a < 0$ and $b > 0$.
Thus, using Theorem~4.1 of \cite{CChavez_Song_2004},
the following result is established.

\begin{theorem}
The endemic equilibrium $\Sigma_H$ is locally asymptotically stable
for the basic reproduction number $R_1$ \eqref{eq:R0:HIV:only} near 1.
\end{theorem}


\subsection{TB-only model}

The sub-model of \eqref{model:TB:HIV:Chronic:semcont}
with no HIV/AIDS disease, that is,
$I_H = A = C_H = L_{TH} = I_{TH}= R_{H} = A_T = 0$, is given by
\begin{equation}
\label{model:TB:only}
\begin{cases}
\dot{S}(t) = \Lambda - \lambda_T(t) S(t) - \mu S(t),\\[0.2 cm]
\dot{L}_T(t) = \lambda_T(t) S(t) + \beta^{'}_1 \lambda_T(t) R(t) - (k_1 + \tau_1 + \mu)L_T(t),\\[0.2 cm]
\dot{I}_T(t) = k_1 L_T(t) - (\tau_2 +d_T +\mu)I_T(t), \\[0.2 cm]
\dot{R}(t) = \tau_1 L_T(t) + \tau_2 I_T(t) - (\beta^{'}_1 \lambda_T(t) + \mu) R(t),
\end{cases}
\end{equation}
where
\begin{equation*}
\lambda_T(t) = \frac{\beta_1 I_T(t)}{N(t)}
\end{equation*}
and
\begin{equation*}
N(t) = S(t) + L_T(t) + I_T(t) + R(t).
\end{equation*}
The sub-model \eqref{model:TB:only} was proposed and analyzed in \cite{Castillo_Chavez_1997}.
This model incorporates the basic properties of TB transmission and dynamics.
The basic reproduction number $R_2$ of \eqref{model:TB:only} is given by
\begin{equation}
\label{eq:brn:R2}
R_2 = \frac{\Lambda}{\mu N} \left(\frac{\beta_1}{\mu + d_T + \tau_2} \right)
\left( \frac{k_1}{\mu + k_1 + \tau_1}\right).
\end{equation}
The existence, uniqueness and local asymptotic stability of the disease-free
and endemic equilibria are proven in \cite[Theorem~1]{Castillo_Chavez_1997}.


\section{Analysis of the full model}
\label{sec:fullmodel}

We now consider the full model \eqref{model:TB:HIV:Chronic:semcont},
with the DFE given by
\begin{equation*}
\begin{split}
\Xi_0 &= \left( S^\diamond, L_T^\diamond, I_T^\diamond, R^\diamond, I_H^\diamond, A^\diamond,
C_H^\diamond, L_{TH}^\diamond, I_{TH}^\diamond, R_H^\diamond, A_T^\diamond \right)\\
&= \left( \frac{\Lambda}{\mu}, 0, 0, 0, 0, 0, 0, 0, 0, 0, 0 \right).
\end{split}
\end{equation*}
The associated matrices $F$ and $V$ (see Section~\ref{sec:DFE:HIVmodel})
are given, respectively, by
\begin{equation*}
F = \left[ \begin {array}{cc}
F_1 & F_2
\end {array} \right]
\end{equation*}
with
\begin{equation*}
F_1 = \left[ \begin {array}{cccccc}
0&0&0&0&0&0\\ \noalign{\medskip}
\lambda_T &0&{\frac{\beta_1(S + R \beta'_1)}{N}} & \beta'_1 \lambda_T &0&0\\
\noalign{\medskip} 0&0&0&0&0&0\\
\noalign{\medskip}0&0&0&0&0&0\\ \noalign{\medskip}
\lambda_H &0&0&\lambda_H&{\frac{\beta_2(S + R)}{N}}
&{\frac{\beta_2 \eta_A(S + R)}{N}} \\
\noalign{\medskip} 0&0&0&0&0&0\\
\noalign{\medskip} 0&0&0&0&0&0\\
\noalign{\medskip} 0&0&{\frac {\beta'_2 \,\beta_1 R_H}{N}} &0&0&0\\
\noalign{\medskip} 0&0&\delta \lambda_H
+{\frac{\psi\,\beta_1 I_H}{N}}&0&
{\frac{\delta \beta_2 I_T}{N}} +\psi \lambda_T
&{\frac {\delta \beta_2 \eta_A I_T}{N}}\\
\noalign{\medskip} 0&0&0&0&0&0\\
\noalign{\medskip} 0&0&0&0&0&0
\end{array} \right],
\end{equation*}
\begin{equation*}
F_2 = \left[ \begin {array}{ccccc}
0&0&0&0&0\\ \noalign{\medskip}
0&0&{\frac{\beta_1 (S + R \beta'_1)}{N}}
&0&{\frac{\beta_1(S+ R \beta'_1)}{N}}
\\ \noalign{\medskip}
0&0&0&0&0\\ \noalign{\medskip}
0&0&0&0&0\\ \noalign{\medskip}
{\frac{\beta_2 \eta_C(S + R)}{N}}&{\frac {\beta_2(S+R)}{N}}
&{\frac {\beta_2(S+R)}{N}}&{\frac {\beta_2(S+R)}{N}}
&{\frac{\beta_2 \eta_A (S+R)}{N}}\\ \noalign{\medskip}
0&0&0&0&0\\ \noalign{\medskip}
0&0&0&0&0\\ \noalign{\medskip}
0&{\frac{\beta'_2 \,\beta_1 R_H}{N}} &\beta'_2 \lambda_T
&{\frac{\beta'_2 \,\beta_1 R_H}{N}}&0\\ \noalign{\medskip}
{\frac{\delta \beta_2 I_T}{N}}&{\frac {\delta \beta_2 I_T}{N}}
+{\frac {\psi \beta_1 I_H}{N}}&\frac{\delta \beta_2 I_T}{N}
&{\frac{\delta \beta_2 \eta_A I_T}{N}}+{\frac {\psi \beta_1 I_H}{N}}&0\\
\noalign{\medskip} 0&0&0&0&0\\
\noalign{\medskip} 0&0&0&0&0
\end{array} \right],
\end{equation*}
and
$V = \left[ V_1 \quad V_2\right]$ with
\begin{equation*}
V_1 =  \left[ \begin {array}{cccccc}
\lambda_T + \lambda_H +\mu&0& {\frac{\beta_1 S}{N}}&0&{\frac{\beta_2 S}{N}}
&{\frac{\beta_2 S \eta_A}{N}}\\ \noalign{\medskip}
0& C_4 &0&0&0&0\\ \noalign{\medskip}
0&-k_1 &\delta \lambda_H + C_5 &0&
{\frac{\delta \beta_2 I_T}{N}}&{\frac{\delta \beta_2 \eta_A I_T}{N}}\\
\noalign{\medskip} 0&-\tau_1 & \frac{\beta'_1 \beta_1 R-\tau_2 N}{N}&\beta'_1 \lambda_T
+ \lambda_H +\mu&{\frac{\beta_2 R}{N}}&{\frac{\beta_2 \eta_A R}{N}}\\
\noalign{\medskip} 0&0&{\frac {\psi \beta_1 I_H}{N}}&0&\psi \lambda_T + C_1 &-\alpha_1\\
\noalign{\medskip} 0&0&0&0&-\rho_1 & C_2\\
\noalign{\medskip} 0&0&0&0&-\phi&0\\
\noalign{\medskip} 0&0&0&0&0&0\\
\noalign{\medskip} 0&0&0&0&0&0\\
\noalign{\medskip} 0&0&{\frac{\beta'_2 \beta_1 R_H}{N}}&0&0&0\\
\noalign{\medskip} 0&0&0&0&0&0
\end {array} \right]
\end{equation*}
and
\begin{equation*}
V_2 =  \left[ \begin {array}{ccccc}
{\frac{\beta_2 \eta_C S}{N}}&{\frac{\beta_2 S}{N}} &
{\frac{\left(\beta_1 + \beta_2\right)S}{N}}&{\frac{\beta_2 S}{N}}&
{\frac{\left(\beta_1 + \beta_2 \eta_A\right)S}{N}} \\ \noalign{\medskip}
0&0&0&0&0\\ \noalign{\medskip}
{\frac {\delta \beta_2 \eta_C I_T}{N}} &{\frac{\delta \beta_2 I_T}{N}}
&{\frac {\delta \beta_2 I_T}{N}}&{\frac{\delta \beta_2 I_T}{N}}&
{\frac {\delta \beta_2 \eta_A I_T}{N}}\\ \noalign{\medskip}
\frac{\beta_2 \eta_C R}{N}&{\frac{\beta_2 R}{N}}&
{\frac{\left(\beta'_1 \beta_1+\beta_2\right)R}{N}}&
{\frac{\beta_2 R}{N}}& \frac{\left(\beta'_1 \beta_1 + \beta_2 \eta_A\right)R}{N} \\
\noalign{\medskip} -\omega_1 &0&{\frac{\psi \beta_1 I_H}{N}}&0&
\frac{\psi \beta_1 I_H}{N}\\ \noalign{\medskip}
0&{\frac{\beta_2 R}{N}}&0&
-\omega_2 &0\\ \noalign{\medskip}
\omega_1 +\mu&-r \, \tau_3 &-p \, \rho_2 &0&0\\
\noalign{\medskip} 0& C_6 &0&0&0\\
\noalign{\medskip} 0&-k_2 & C_7 &0&-\alpha_2\\
\noalign{\medskip} 0&\left( -1+r \right) \tau_3 &-q \,
\rho_2 + {\frac{\beta'_2 \beta_1 R_H}{N}}&
\beta'_2 \lambda_T +\omega_2 +\mu&{\frac{\beta'_2 \beta_1 R_H}{N}}\\
\noalign{\medskip} 0&0&\left(-1+p+q \right) \rho_2 &0& C_8
\end {array} \right],
\end{equation*}
where
$C_4 = k_1 + \tau_1 + \mu$, $C_5 = \tau_2 + \mu + d_T$,
$C_6 = k_2 + \tau_3 + \mu$, $C_7 = \rho_2 + \mu + d_T$
and $C_8 = \alpha_2 + d_{TA} + \mu$.
The dominant eigenvalues of the matrix $F  V^{-1}$ are
\begin{equation*}
R_1 = \frac {{\beta_2}\,\Lambda\, \left( C_3(C_2 +{\eta_A}\,{\rho_1})
+{\eta_C}\,\phi C_2 \right) }
{N \mu \left[ \mu \left(C_3(\rho_1 + C_2) + C_2 \phi
+ \rho_1 d_A \right) + \rho_1 \omega_1 d_A  \right] },
\quad R_2= \frac{\Lambda \beta_1 k_1}{\mu N C_5 C_4}.
\end{equation*}
Thus, the basic reproduction number $R_0$
of the model \eqref{model:TB:HIV:Chronic:comcontrolos} is given by
$$
R_0 = \max \{ R_1, R_2 \}.
$$

Using the same procedure as in Section~\ref{sec:DFE:HIVmodel}, the following result holds.

\begin{lemma}
The DFE of the full HIV-TB model \eqref{model:TB:HIV:Chronic:semcont},
given by $\Xi_0$, is locally asymptotically stable
if $R_0 < 1$, and unstable if $R_0 > 1$.
\end{lemma}

\begin{remark}
There are different ways to compute the basic reproduction number $R_0$. Here we are computing it using
one of the most well-known methods: $R_0$ is the dominant eigenvalue of the associated
next-generation matrix $F V^{-1}$ (see Remark~\ref{rem3}). A justification for the value
of the basic reproduction number $R_0$ to be $\max \{R_1, R_2 \}$
is given in Section~4.4 of \cite{van:den:Driessche:2002}.
\end{remark}


\section{Optimal control problem}
\label{sec:optimal:control}

In this section we present an optimal control problem, describing our goal
and the restrictions of the epidemic. In the model without controls discussed so far,
we have $p$ representing the \emph{fraction of $I_{TH}$ individuals that take HIV and TB treatment}
and $q$ representing the \emph{fraction of $I_{TH}$ individuals that take TB treatment only}.
Roughly speaking, the problem of optimal control consists
to determine the optimal combination for the values of $p$ and $q$. For this reason,
we take $p$ as the control $u_1$ and $q$ as the control $u_2$. Precisely, we add to the model
\eqref{model:TB:HIV:Chronic:semcont} the two control functions
$u_1(\cdot)$ and $u_2(\cdot)$ in the following way:
\begin{equation}
\label{model:TB:HIV:Chronic:comcontrolos}
\begin{cases}
\dot{S}(t) = \Lambda - \lambda_T(t) S(t) - \lambda_H(t) S(t) - \mu S(t),\\[0.2 cm]
\dot{L}_T(t) = \lambda_T(t) S(t) + \beta^{'}_1 \lambda_T(t) R(t)
- \left(k_1 + \tau_1 + \mu\right) L_T(t),\\[0.2 cm]
\dot{I}_T(t) = k_1 L_T(t) - \left(\tau_2 +d_T +\mu + \delta \lambda_H(t)\right)I_T(t), \\[0.2 cm]
\dot{R}(t) = \tau_1 L_T(t) + \tau_2 I_T(t) - (\beta^{'}_1 \lambda_T(t)
+ \lambda_H(t) + \mu) R(t),\\[0.2 cm]
\dot{I}_H(t) = \lambda_H(t) S(t) - (\rho_1  + \phi + \psi \lambda_T(t) + \mu)I_H(t)
+ \alpha_1 A(t) + \lambda_H(t) R(t) + \omega_1 C_H(t), \\[0.2 cm]
\dot{A}(t) =  \rho_1 I_H(t) + \omega_2 R_H(t)
- \alpha_1 A(t) - (\mu + d_A) A(t),\\[0.2 cm]
\dot{C}_H(t) = \phi I_H(t) + u_1(t) \, \rho_2 I_{TH}(t)
+ r\,  \tau_3 L_{TH}(t) - (\omega_1 + \mu)C_H(t),\\[0.2 cm]
\dot{L}_{TH}(t) = \beta^{'}_2 \lambda_T(t) R_{H}(t)
- \left(k_2 + \tau_3 + \mu\right) L_{TH}(t),\\[0.2 cm]
\dot{I}_{TH}(t) = \delta \lambda_H(t) I_T(t) + \psi \lambda_T(t) I_H(t)
+ \alpha_2 A_T(t)+ k_2 L_{TH}(t)
- \left( \rho_2 + \mu + d_T \right)I_{TH}(t),\\[0.2 cm]
\dot{R}_{H}(t) = u_2(t) \rho_2 I_{TH}(t) + (1-r)\, \tau_3 L_{TH}(t)
- \left(\beta^{'}_2 \lambda_T(t) + \omega_2 + \mu\right)R_{H}(t), \\[0.2 cm]
\dot{A}_T(t) = \left(1-(u_1(t)+u_2(t))\right)\rho_2 I_{TH}(t)
-\left(\alpha_2 + \mu + d_{TA}\right)A_T(t).
\end{cases}
\end{equation}
As already mentioned, the controls $u_1$ and $u_2$ represent the fraction of $I_{TH}$ individuals that are treated
for TB and HIV (simultaneously) and treated for TB only, respectively. If we consider fixed values
for $u_1$ and $u_2$ in \eqref{model:TB:HIV:Chronic:comcontrolos}, then we get the model
\eqref{model:TB:HIV:Chronic:semcont} with $u_1 = p$ and $u_2 = q$.
The aim is to find the optimal values $u_1^*$ and $u_2^*$ of the controls
$u_1$ and $u_2$, such that the associated state trajectories
$S^*, L_T^*, I_T^*, R^*, I_H^*, A^*, C_H^*, L_{TH}^*, I_{TH}^*, R_H^*, A_T^*$,
solution of the system \eqref{model:TB:HIV:Chronic:comcontrolos}
in the time interval $[0, T]$ with initial conditions
$S^*(0)$, $L_T^*(0)$, $I_T^*(0)$, $R^*(0)$, $I_H^*(0)$,
$A^*(0)$, $C_H^*(0)$, $L_{TH}^*(0)$, $I_{TH}^*(0)$, $R_H^*(0)$, $A_T^*(0)$,
minimize the objective functional.
Here the objective functional considers the number of HIV-infected individuals
with AIDS symptoms co-infected with TB $A_T$, and the implementation cost
of the strategies associated to the controls $u_i$, $i=1, 2$.
The controls are bounded between $0$ and $0.95$.
We assume that $p$ and $q$ cannot take values greater than $0.95$
because we assume that there are some budgetary constraints or some resistance
from patients in making the treatments (treatment for HIV and TB together or just the treatment for TB).
In other words, we assume that one cannot treat all the people for both diseases
or even just for tuberculosis. This is more than reasonable from biological side.
Moreover, the sum of $p + q$ is also taken as bounded by 0.95. This is related with the formulation of the model.
Indeed, note that $1- (p + q)$ is the fraction of $I_{TH}$ individuals who are not treated for TB and HIV simultaneously
and are also not treated for TB alone. For this reason, what we assume is that this fraction of individuals
takes at least the value of 5\%. This is in agreement with available medical data.
Precisely, we consider the state system \eqref{model:TB:HIV:Chronic:comcontrolos}
of ordinary differential equations in $\mathbb{R}^{11}$
with the set of admissible control functions given by
\begin{multline}
\label{eq:adm:controls}
\Theta = \biggl\{ (u_1(\cdot), u_2(\cdot)) \in (L^{\infty}(0, T))^2 \,
| \,  0 \leq u_1 (t), u_2(t) \leq 0.95 \, \\
\text{ and } \, 0 \leq u_1 (t)
+ u_2(t) \leq 0.95 ,  \, \forall \, t \in [0, T] \, \biggr\} .
\end{multline}
The objective functional is given by
\begin{equation}
\label{costfunction}
J(u_1(\cdot), u_2(\cdot)) = \int_0^{T} \left[ A_T(t)
+ \frac{W_1}{2}u_1^2(t) + \frac{W_2}{2}u_2^2(t) \right] dt,
\end{equation}
where the constants $W_1$ and $W_2$ are a measure
of the relative cost of the interventions
associated to the controls $u_1$ and $u_2$, respectively.

\begin{remark}
Epidemiologically, our cost functional tells us that we want to minimize the number
of HIV-infected individuals with AIDS symptoms co-infected with active TB.
For that, one applies control measures that are associated with some implementation
costs that we also intend to minimize. Other cost functionals may be used as well.
Here, by considering the cost with controls in a quadratic form,
we are being consistent with previous works in the literature
(see, e.g., \cite{MR3266821,Silva:Torres:TBOC:MBS:2013}).
Moreover, a quadratic structure in the control has
mathematical advantages: if the control set is a compact and convex polyhedron (as
it is the case here), it imply that the Hamiltonian attains its minimum over
the control set at a unique point. For future work we plan
to compare the results now obtained, for a cost with a quadratic
form in the controls, with those of a linear cost in the controls.
\end{remark}

In order to simplify the formulation of the optimal control problem,
let $f_i$ represent the right-hand side of the $i$th equation of system
\eqref{model:TB:HIV:Chronic:comcontrolos}, $x_i$ be the state variable
whose derivative is given by the $i$th component of $F$,
$F = \left(f_1, \ldots, f_{11}\right)$, and
$X = \left(x_1, \ldots, x_{11}\right)$, $i = 1, \ldots, 11$.
We consider the optimal control problem of determining
$X^*(\cdot)$ associated to an admissible control pair
$\left(u_1^*(\cdot), u_2^*(\cdot) \right) \in \Theta$
on the time interval $[0, T]$, satisfying
\eqref{model:TB:HIV:Chronic:comcontrolos},
the initial conditions $X(0)$ and
minimizing the cost function \eqref{costfunction}, that is,
\begin{equation}
\label{mincostfunct}
J(u_1^*(\cdot), u_2^*(\cdot))
= \min_{\Theta} J(u_1(\cdot), u_2(\cdot)).
\end{equation}
The existence of optimal controls $\left(u_1^*(\cdot), u_2^*(\cdot)\right)$
comes from the convexity of the cost functional \eqref{costfunction}
with respect to the controls and the regularity of the system
\eqref{model:TB:HIV:Chronic:comcontrolos} (see, \textrm{e.g.},
\cite{Cesari_1983,Fleming_Rishel_1975}
for existence results of optimal solutions).
According to the Pontryagin maximum principle \cite{Pontryagin_et_all_1962},
if $\left(u_1^*(\cdot), u_2^*(\cdot)\right) \in \Theta$ is optimal for the problem
\eqref{model:TB:HIV:Chronic:comcontrolos}, \eqref{mincostfunct} with the initial conditions
$X(0)$ and fixed final time $T$, then there exists
a nontrivial absolutely continuous mapping $\lambda : [0, T] \to \mathbb{R}^{11}$,
$\lambda(t) = \left(\lambda_1(t), \ldots, \lambda_{11}(t)\right)$,
called \emph{adjoint vector}, such that
\begin{equation}
\label{adjsystemPMP}
\dot{x}_i = \frac{\partial H}{\partial \lambda_i}(X,\lambda,u_1, u_2),
\quad \dot{\lambda}_i = -\frac{\partial H}{\partial x_i}(X,\lambda,u_1, u_2),
\quad i = 1, \ldots, 11,
\end{equation}
where the function $H= H(X,\lambda,u_1, u_2)$ defined by
\begin{equation*}
H= A_T  + \frac{W_1}{2}u_1^2 + \frac{W_2}{2}u_2^2 + \langle\lambda,F(X,u_1, u_2)\rangle\\
\end{equation*}
is called the \emph{Hamiltonian}, and the minimality condition
\begin{equation}
\label{maxcondPMP}
H(X^*(t), \lambda^*(t), u_1^*(t), u_2^*(t))
= \min_{\stackrel{0 \leq u_1, u_2 \leq 0.95}{u_1 + u_2 \le 0.95}} H(X^*(t),\lambda^*(t), u_1, u_2)
\end{equation}
holds almost everywhere on $[0, T]$. Moreover, the transversality conditions
\begin{equation*}
\lambda_i(T) = 0, \quad i =1,\ldots, 11,
\end{equation*}
are also satisfied.


\section{Numerical results and discussion}
\label{sec:num:results}

In this section we present results of the numerical implementation
of extremal control strategies for the TB-HIV model
\eqref{model:TB:HIV:Chronic:comcontrolos}.
\begin{table}[!htb]
\centering
\begin{tabular}{c  c  c  c  c  c } \hline\hline
$S(0)$ & $L_T(0)$ & $I_T(0)$ & $R(0)$ & $I_H(0)$ & $A(0)$ \\[0.1cm]
$\frac{66 N(0)}{120}$ & $\frac{37 N(0)}{120}$ & $\frac{5 N(0)}{120}$
& $\frac{2 N(0)}{120}$ & $\frac{2 N(0)}{120}$ & $\frac{N(0)}{120}$ \\[0.2cm] \hline
$C_H(0)$ & $L_{TH}(0)$ & $I_{TH}(0)$ & $R_H(0)$ & $A_T(0)$ & \\[0.1cm]
$\frac{N(0)}{120}$ & $\frac{2 N(0)}{120}$ & $\frac{2 N(0)}{120}$
& $\frac{N(0)}{120}$ & $\frac{N(0)}{120}$ & \\[0.2cm] \hline \hline
\end{tabular}
\caption{Initial conditions of the TB-HIV/AIDS model, where $N(0) = 30000$.}
\label{table:init:cond}
\end{table}
First we solve numerically the optimal control problem \eqref{model:TB:HIV:Chronic:comcontrolos},
\eqref{mincostfunct} with initial conditions given in Table~\ref{table:init:cond}, and fixed final
time $T = 50$ years. The initial conditions were estimated as follows.
We assume that more than half of population ($55\%$) belongs to
the subgroup of susceptible and that a big percentage ($\simeq 31\%$) is infected with TB but is in the latent stage.
This is justified from the fact that ``about one-third of the world's population has latent TB'', as one can find
in the website of the World Health Organization (WHO) \cite{fs:who}.
The value for the fraction of people infected with HIV is assumed $\simeq 1.7\%$,
based on HIV \& AIDS Information from AVERT.org \cite{Worldwide:HIV:AIDS:Stat}:
``\emph{There is either a generalised or concentrated epidemic.
In a generalised epidemic, HIV prevalence is 1\% or more in the general population.
In a concentrated or low level epidemics, HIV prevalence is below 1\% in the general population
but exceeds 5\% in specific at-risk populations like injecting drug users or sex workers,
or HIV prevalence is not recorded at a significant level in any group.}''
The remaining values are estimated by assuming that we are in a ``controlled'' situation,
without large percentages in the groups of highest risk such as $A$, $A_T$ and $C_H$.
Our aim is to find the optimal combination of the fraction of individuals
$I_{TH}$ that take correctly HIV and TB treatment ($u_1^*$) or take only TB treatment ($u_2^*$),
in order to minimize the number of individuals with AIDS and TB diseases $A_T$.
Different approaches were used to obtain and confirm the numerical results. One approach
consisted in using IPOPT (short for ``Interior Point OPTimizer'', a software library
for large scale nonlinear optimization of continuous systems) \cite{MR2195616}
and the algebraic modeling language AMPL (acronym for ``A Mathematical Programming Language'')
\cite{AMPL}. A second approach was to use the PROPT Matlab Optimal Control Software \cite{PROPT}.
For more details we refer the reader to
\cite{Silva:Torres:TBOC:NACO:2012,Silva:Torres:TBOC:MBS:2013},
where the same optimization approaches are used.
In Figure~\ref{fig:CH:ITH:AT:minAT} we compare the extremal dynamics $C_H^*$, $I_{TH}^*$ and $A_T^*$
associated to the extremal controls $u_1^*$ and $u_2^*$ with the dynamics of the model
\eqref{model:TB:HIV:Chronic:comcontrolos} with $u_1(t)=p $ and $u_2(t)=q$, which coincide
with model \eqref{model:TB:HIV:Chronic:semcont}. In this simulations we consider $\beta_1 = 0.6$,
$\beta_2 = 0.1$ and the rest of the parameters take the values of Table~\ref{table:parameters:TB-HIV:Chronic},
which corresponds to $R_0 = 4.91159$ ($R_1 = 4.91159$, $R_2 = 1.07437$).
We assume that the weight constants take the same value $W_1 = W_2 = 50$.
Observe that the number of individuals with AIDS and TB diseases $A_T$
decreases significantly when the control measures $u_1^*$, $u_2^*$ are implemented, see
Figure~\ref{fig:CH:ITH:AT:minAT} (c). On the other hand, the number of individuals that stays
in the class $C_H$ increases in opposition to the number of individuals that have both infections HIV and TB,
see Figure~\ref{fig:CH:ITH:AT:minAT} (a) and (b).
\begin{figure}[!htb]
\centering
\subfloat[\footnotesize{$C_H$}]{\label{CH:minAT}
\includegraphics[width=0.33\textwidth]{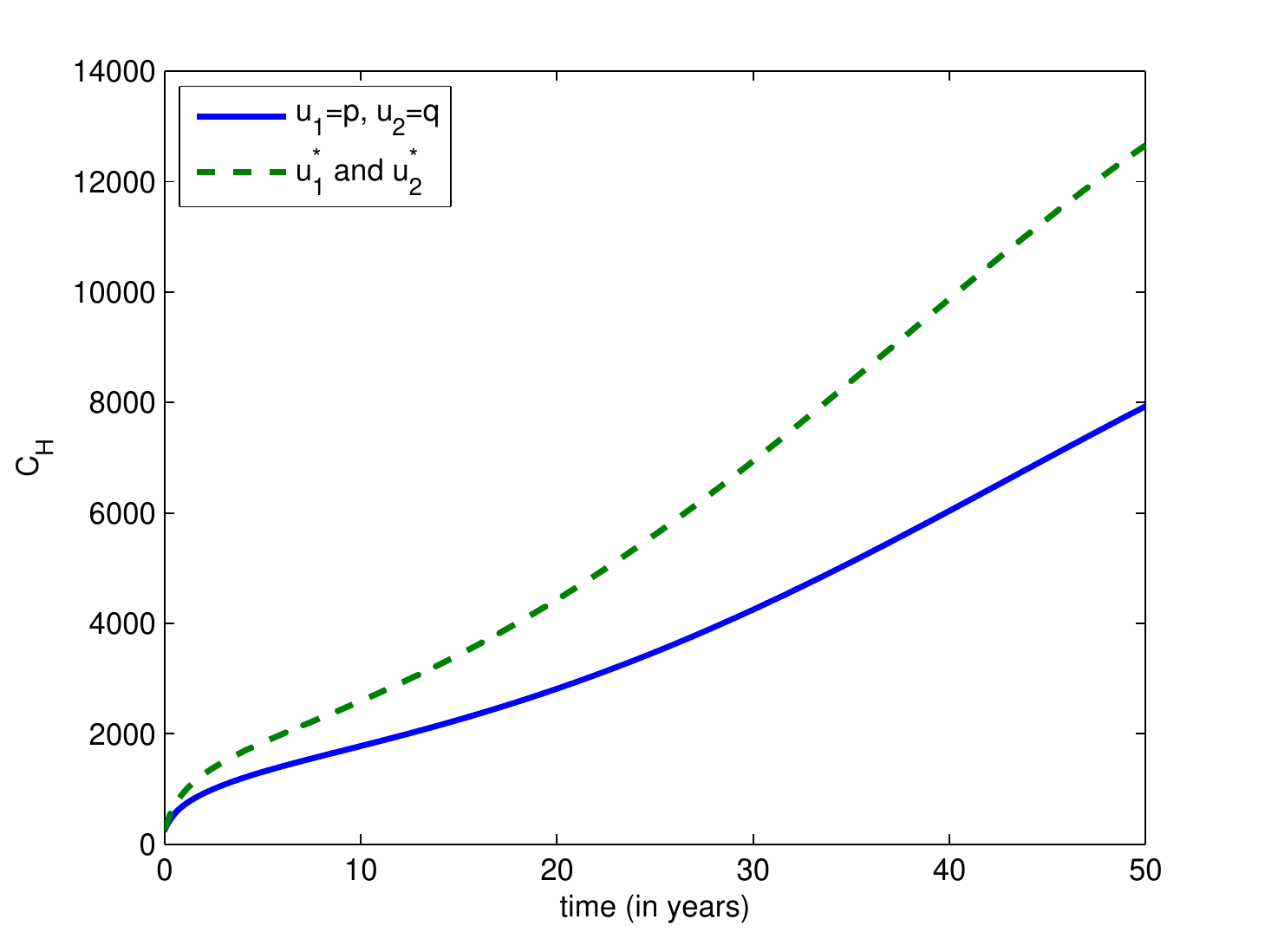}}
\subfloat[\footnotesize{$I_{TH}$}]{\label{ITH:minAT}
\includegraphics[width=0.33\textwidth]{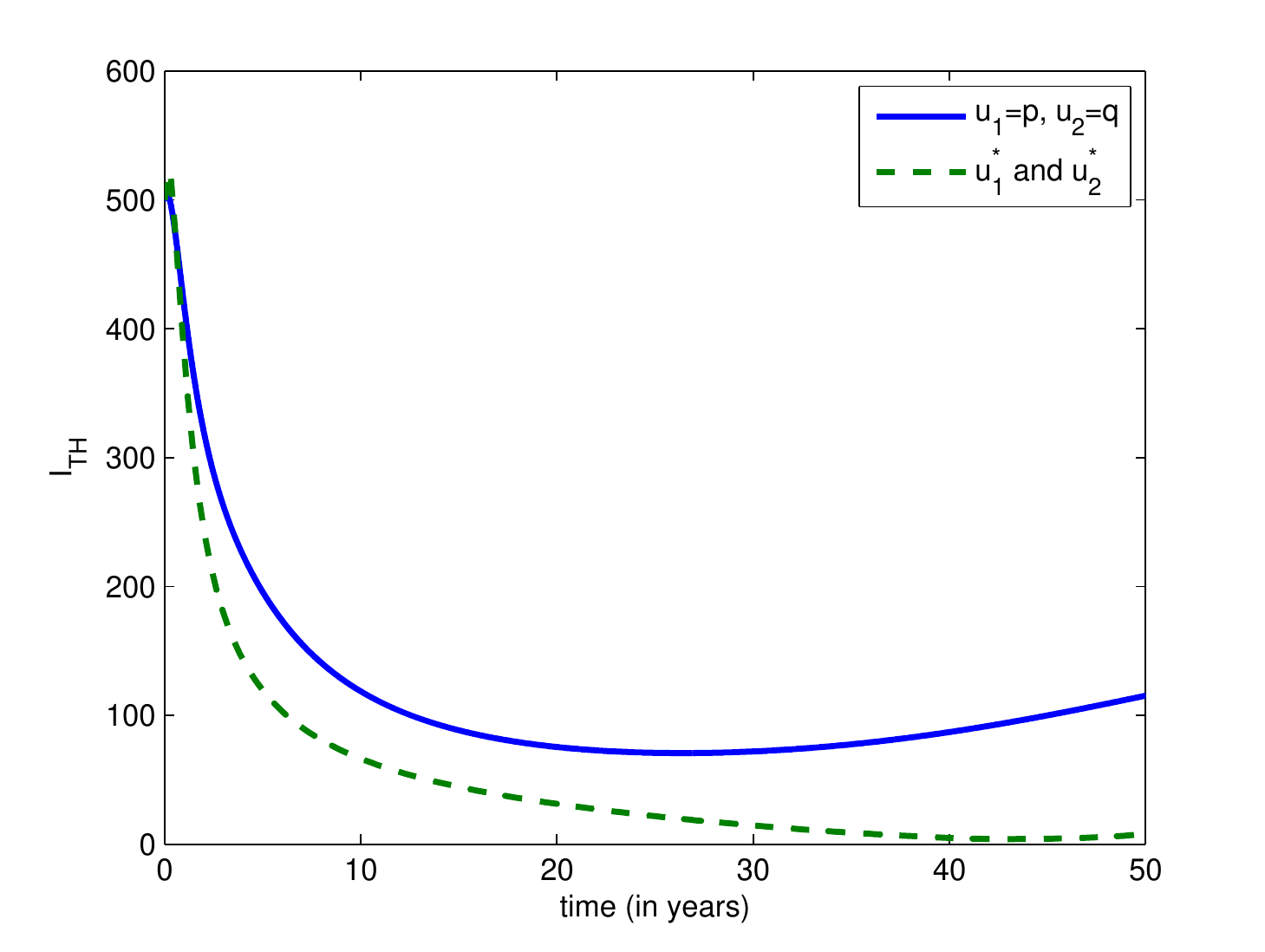}}
\subfloat[\footnotesize{$A_T$}]{\label{AT:minAT}
\includegraphics[width=0.33\textwidth]{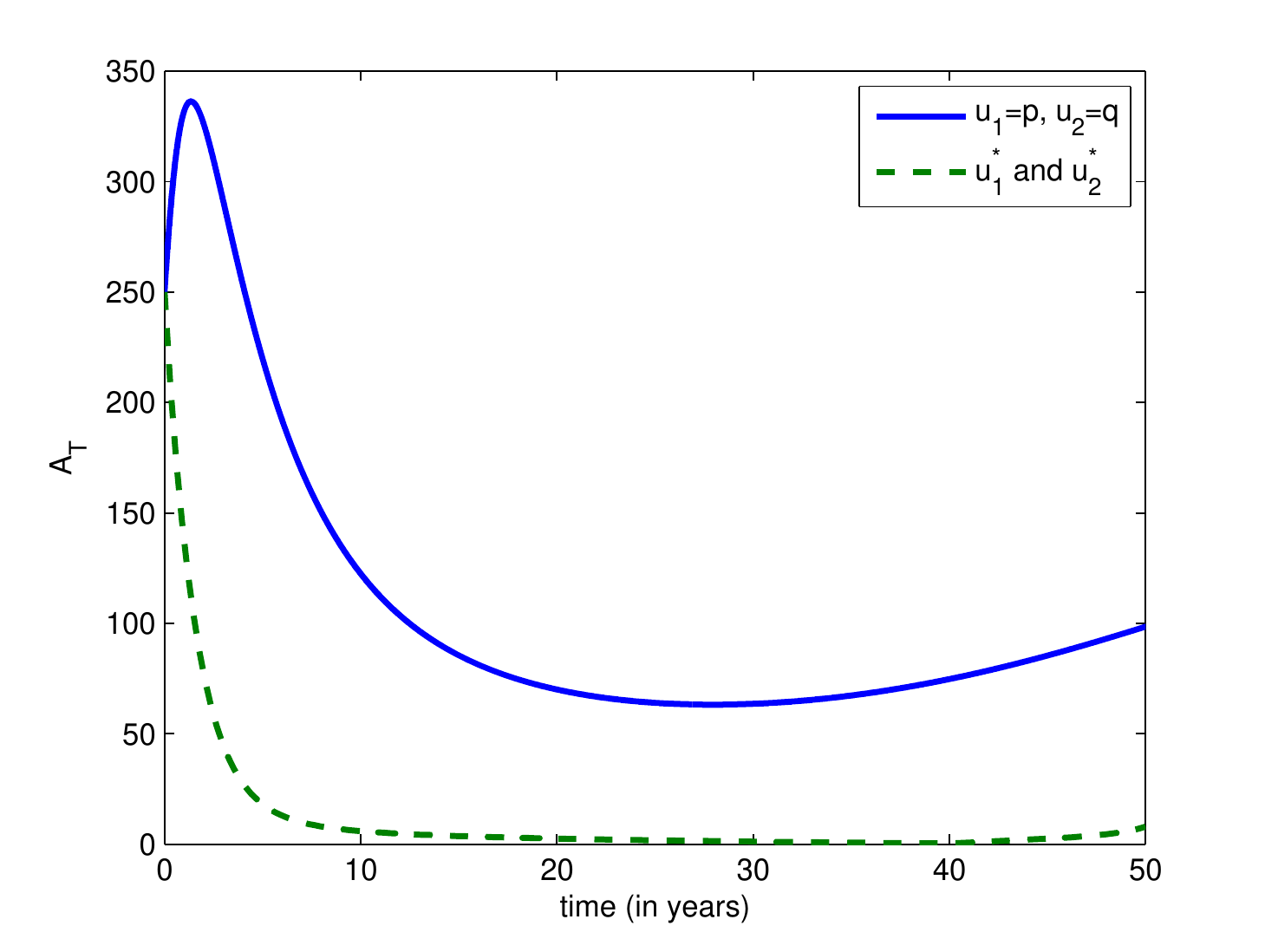}}
\caption{Dynamics $C_H$, $I_{TH}$ and $A_T$ for cost functional \eqref{costfunction},
$\beta_1 = 0.6$, $\beta_2 = 0.1$, $W_1 = W_2 = 50$ and parameter
values from Table~\ref{table:parameters:TB-HIV:Chronic}.}
\label{fig:CH:ITH:AT:minAT}
\end{figure}
During approximately 40 years the optimal combination of the fractions of individuals $I_{TH}$
that take HIV and TB treatments simultaneously and only TB treatment is around $0.5$ and $0.46$,
respectively, see Figure~\ref{fig:u1:u2:minAT} (a) and (b). In Figure~\ref{fig:u1:u2:minAT} (c)
we observe that the extremal controls satisfy the restriction \eqref{eq:adm:controls}.
\begin{figure}[!htb]
\centering
\subfloat[\footnotesize{$u_1^*$}]{\label{u1:minAT}
\includegraphics[width=0.33\textwidth]{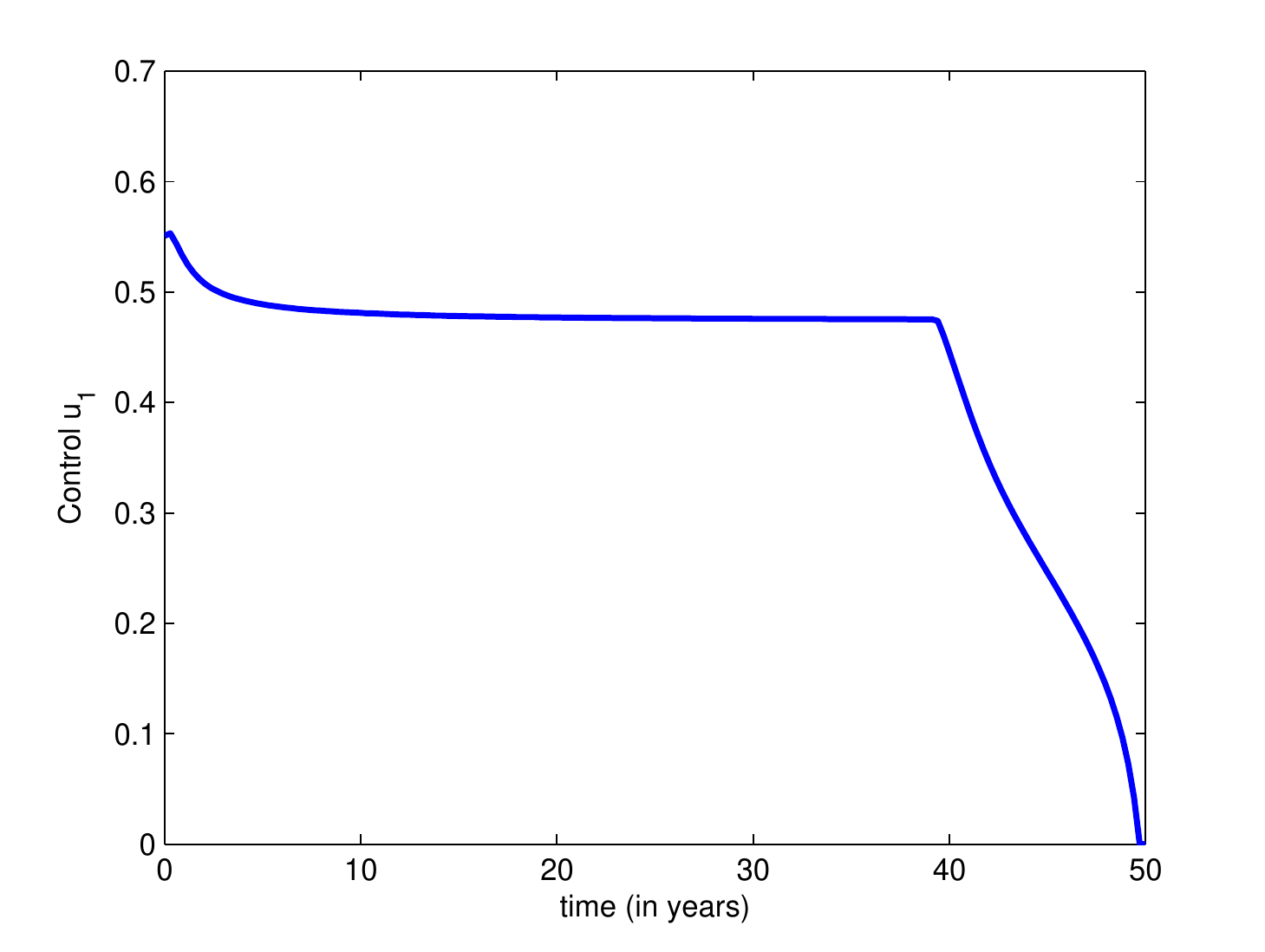}}
\subfloat[\footnotesize{$u_2^*$}]{\label{u2:minATa}
\includegraphics[width=0.33\textwidth]{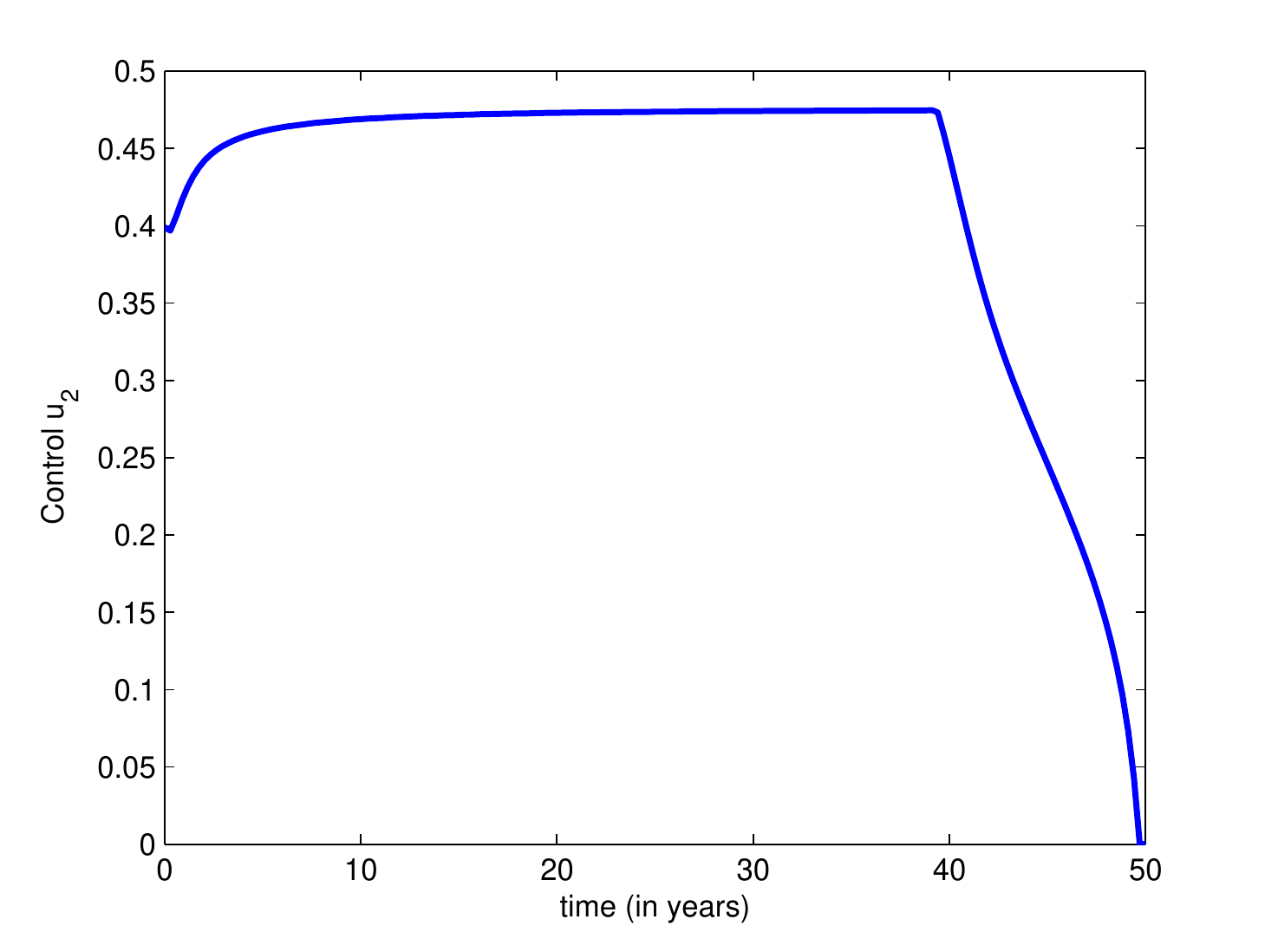}}
\subfloat[\footnotesize{$u_1^*+u_2^*$}]{\label{u2:minATb}
\includegraphics[width=0.33\textwidth]{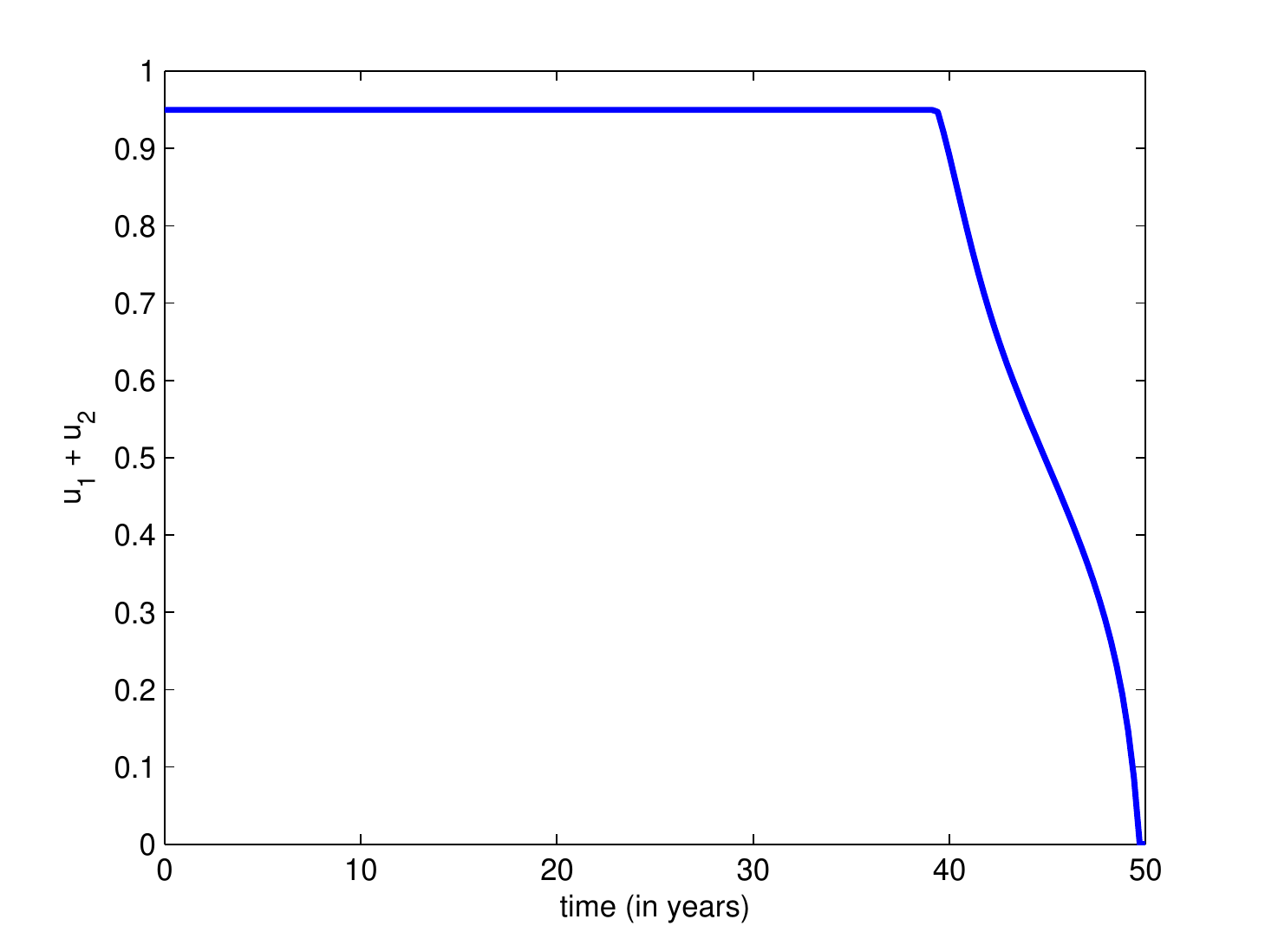}}
\caption{Extremal controls $u_1^*$ and $u_2^*$ for cost functional \eqref{costfunction},
$\beta_1 = 0.6$, $\beta_2 = 0.1$, $W_1 = W_2 = 50$ and parameter values
from Table~\ref{table:parameters:TB-HIV:Chronic}.}
\label{fig:u1:u2:minAT}
\end{figure}
\begin{figure}[!htb]
\centering
\includegraphics[width=0.33\textwidth]{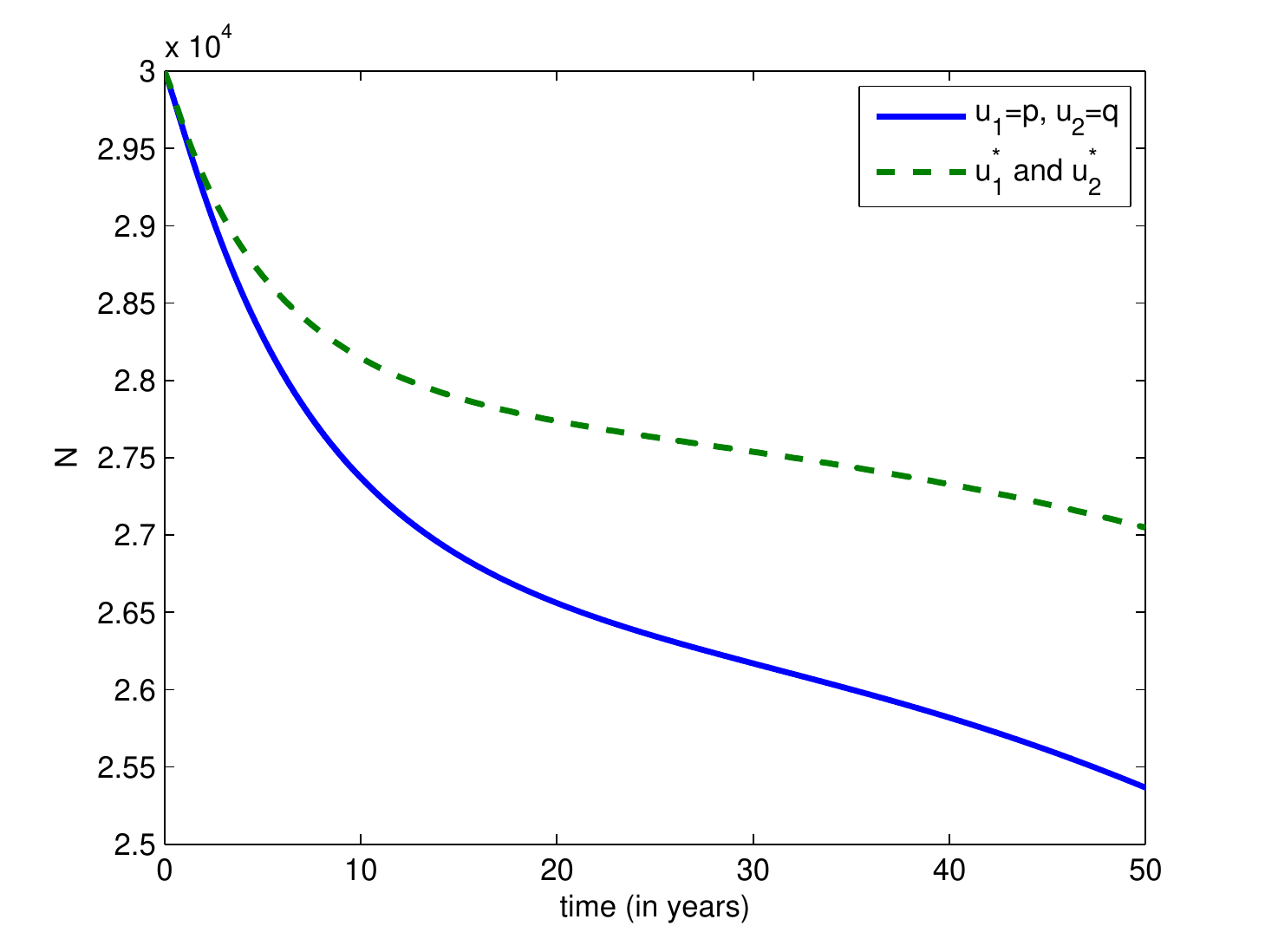}
\caption{Total population $N$ for cost functional \eqref{costfunction},
$\beta_1 = 0.6$, $\beta_2 = 0.1$, $W_1 = W_2 = 50$ and parameter values
from Table~\ref{table:parameters:TB-HIV:Chronic}.}
\label{N:minAT}
\end{figure}
At the end of 50 years, the number of TB and AIDS induced deaths reduces 5\% when the controls
$u_1^*$, $u_2^*$ are applied, see Figure~\ref{N:minAT}.
Since the HIV treatments have higher costs than TB treatment, we can consider that the weight
constant $W_1$ associated to the control $u_1$ takes greater values than $W_2$. In this case,
the fraction of individuals that take TB and HIV treatment $u_1^*$ decreases and the fraction
of individuals that take only TB treatment increases, compared to the previous case $W_1 = W_2 = 50$,
but the associated extremal dynamics $C_H^*$, $I_H^*$ and $A_T^*$ behave similarly to the ones
in the case $W_1 = W_2 = 50$, see Figure~\ref{fig:u1:u2:minAT:W1:500}.
\begin{figure}[!htb]
\centering
\subfloat[\footnotesize{$u_1^*$}]{\label{u1:minAT:W1:500}
\includegraphics[width=0.33\textwidth]{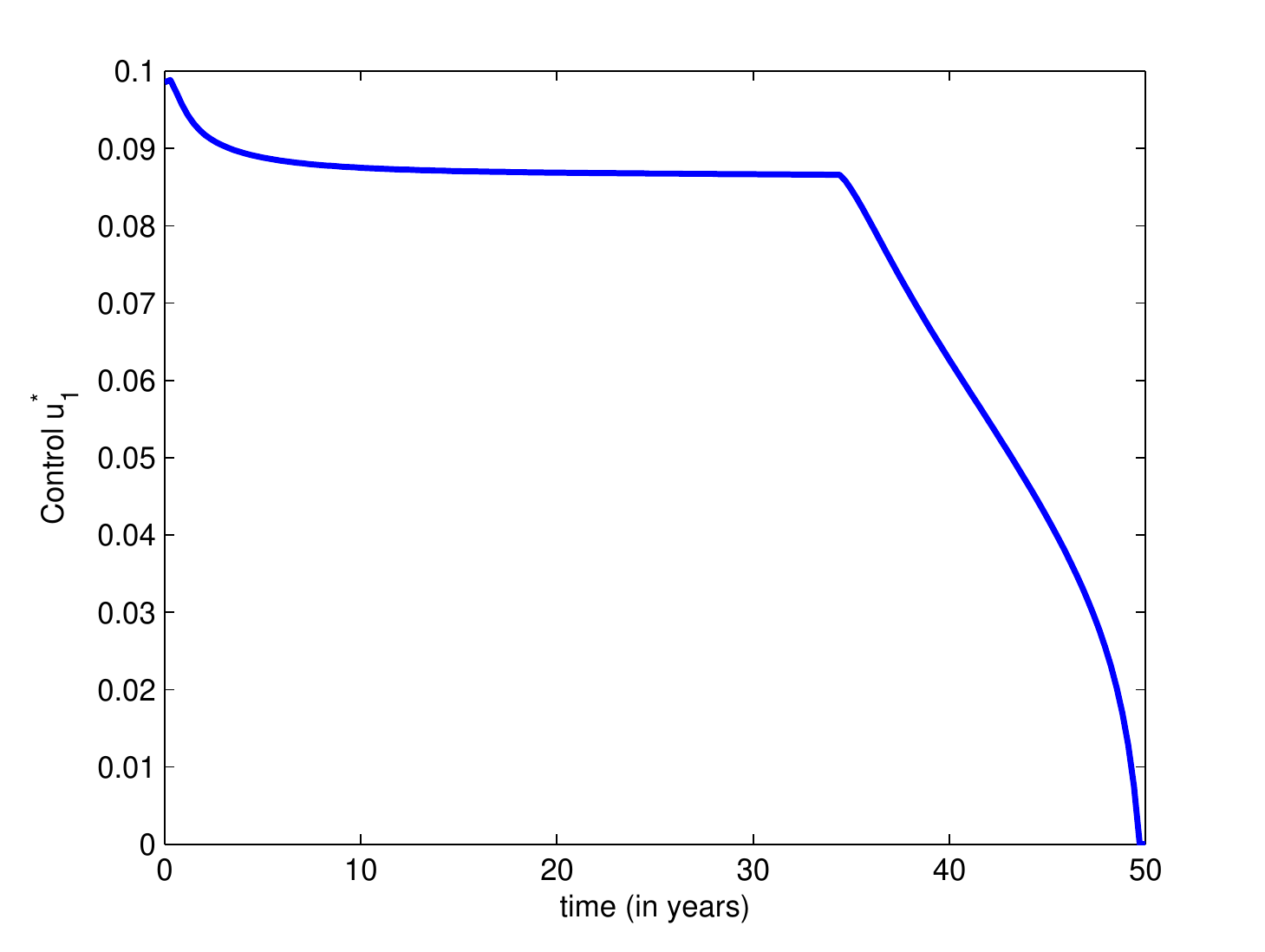}}
\subfloat[\footnotesize{$u_2^*$}]{\label{u2:minAT:W1:500a}
\includegraphics[width=0.33\textwidth]{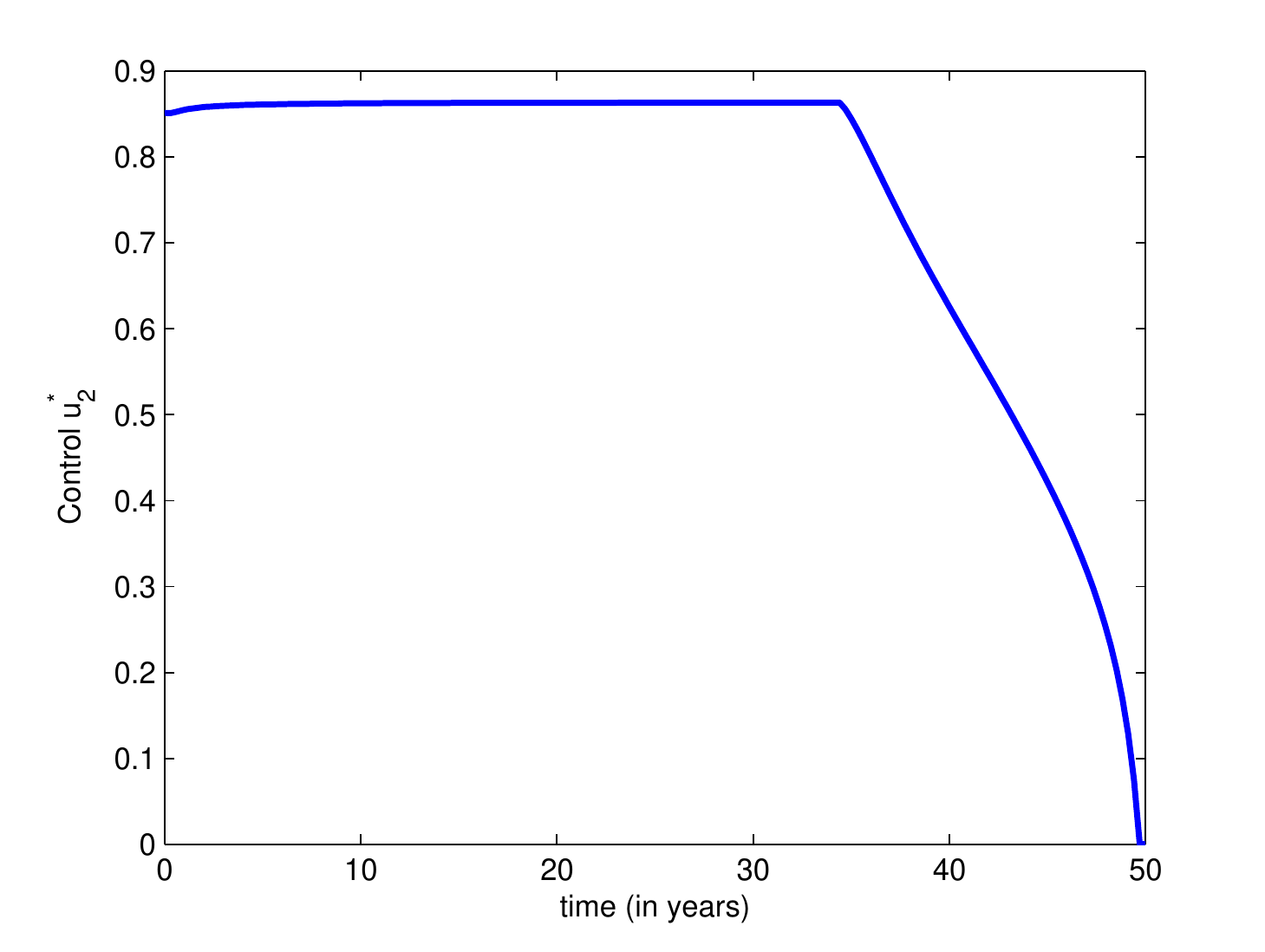}}
\subfloat[\footnotesize{$u_1^*+u_2^*$}]{\label{u2:minAT:W1:500b}
\includegraphics[width=0.33\textwidth]{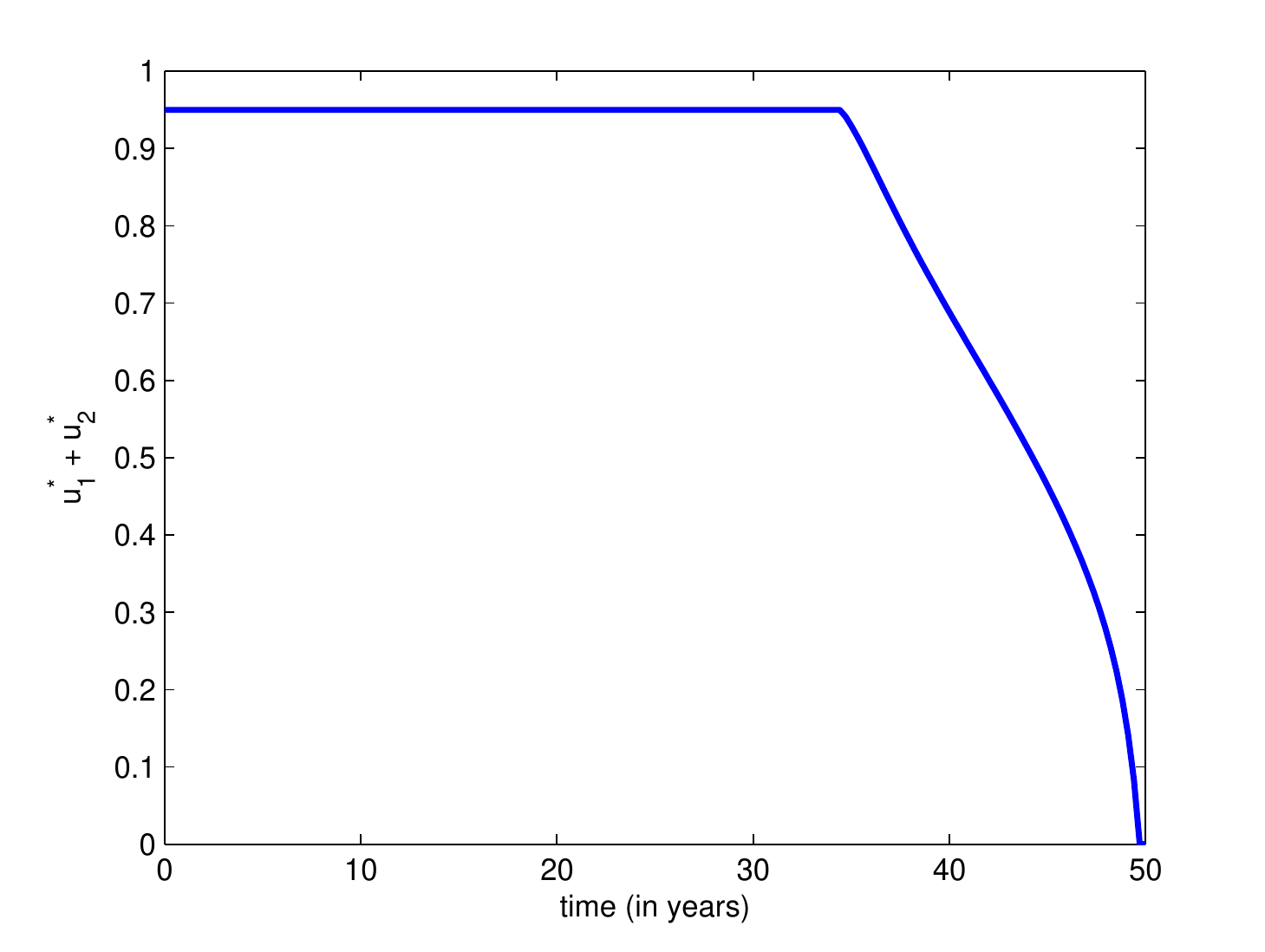}}
\caption{Extremal controls $u_1^*$ and $u_2^*$ for cost functional \eqref{costfunction},
$\beta_1 = 0.6$, $\beta_2 = 0.1$, $W_1 =500$, $W_2 = 50$ and parameter values
from Table~\ref{table:parameters:TB-HIV:Chronic}.}
\label{fig:u1:u2:minAT:W1:500}
\end{figure}
The extremal controls in Figure~\ref{fig:u1:u2:minAT:W1:500} (a) and (b) are not intuitive,
since the fraction of individuals that take both HIV and TB treatments is very low.
If we assume that our aim is to minimize the cost functional
\begin{equation}
J_1(u_1(\cdot), u_2(\cdot)) = \int_0^{T} \left[ A(t) + A_T(t)
+ \frac{W_1}{2}u_1^2(t) + \frac{W_2}{2}u_2^2(t) \right] dt,
\end{equation}
with $T = 10$ years and no disease induced deaths ($d_T = d_{TA} = d_A = 0$), that is,
we wish to minimize the number of individuals that have only AIDS $A$ and have both AIDS
and TB diseases $A_T$, the extremal controls behave in a more intuitive way. Since we assume
that there is no disease induced deaths we consider that is more adequate to consider $T = 10$
instead of $T = 50$ years. Moreover, is this case, the total population is constant.
We observe that the fraction of individuals that take both HIV and TB treatment $u_1^*$
takes the maximum value for more than 7 years, and during this time the extremal control
$u_2^*$ vanishes, see Figure~\ref{fig:u1:u2:minAeAT:nodeath}.
\begin{figure}[!htb]
\centering
\subfloat[\footnotesize{$u_1^*$}]{\label{u1:minAeAT:nodeath}
\includegraphics[width=0.33\textwidth]{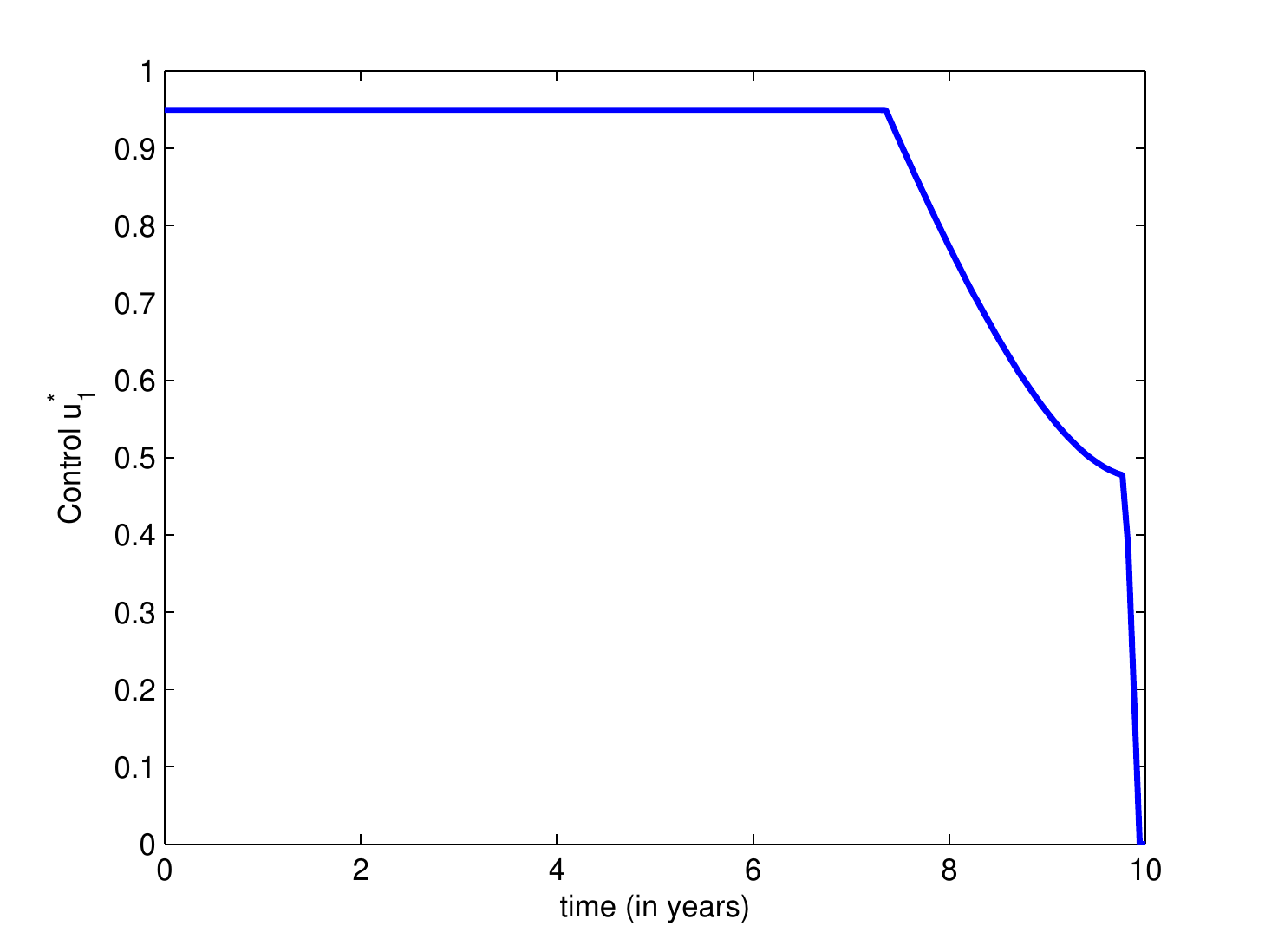}}
\subfloat[\footnotesize{$u_2^*$}]{\label{u2:minAeAT:nodeath}
\includegraphics[width=0.36\textwidth]{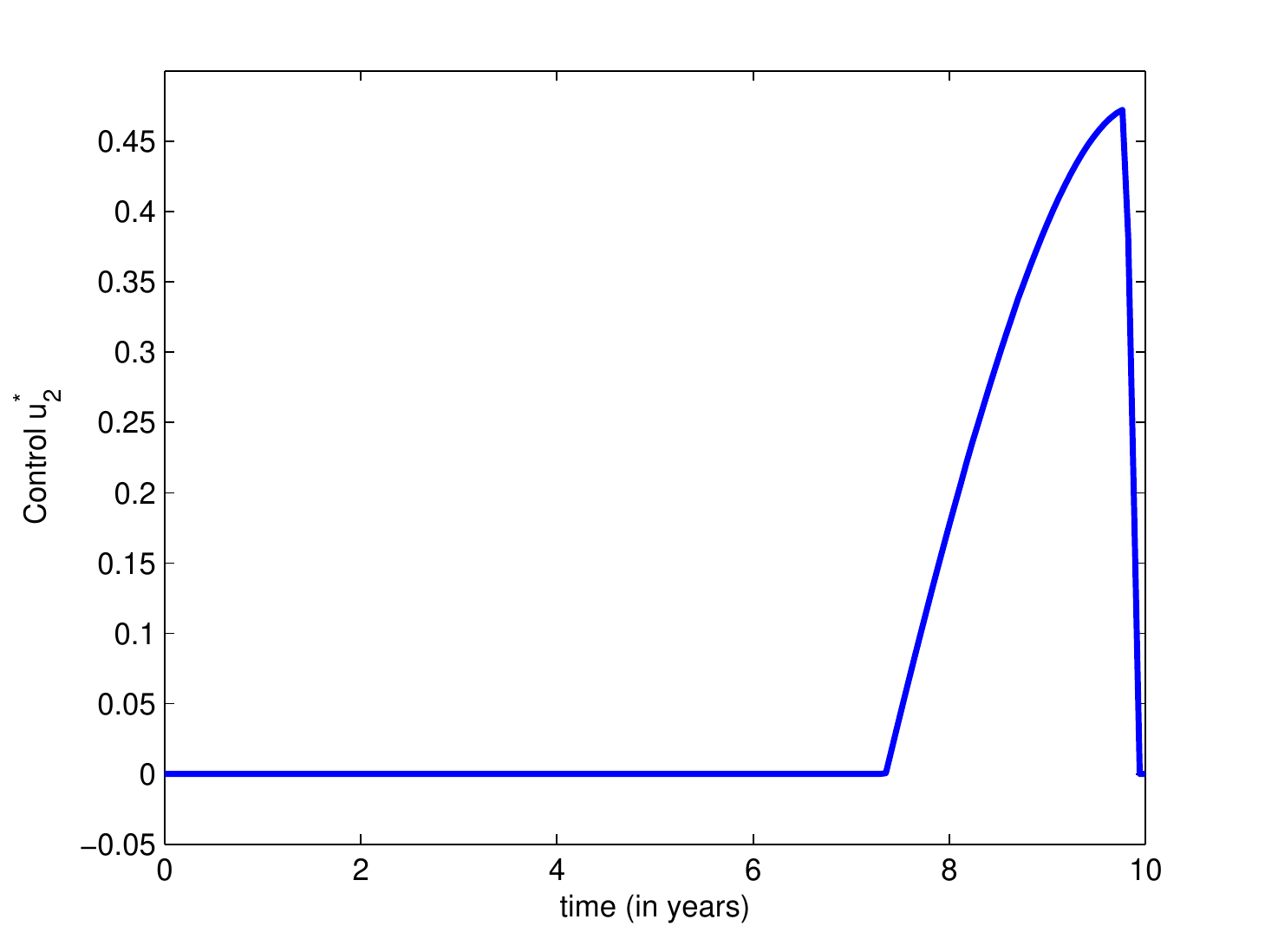}}
\subfloat[\footnotesize{$u_1^*+u_2^*$}]{\label{u1maisu2:minAeAT:nodeath}
\includegraphics[width=0.33\textwidth]{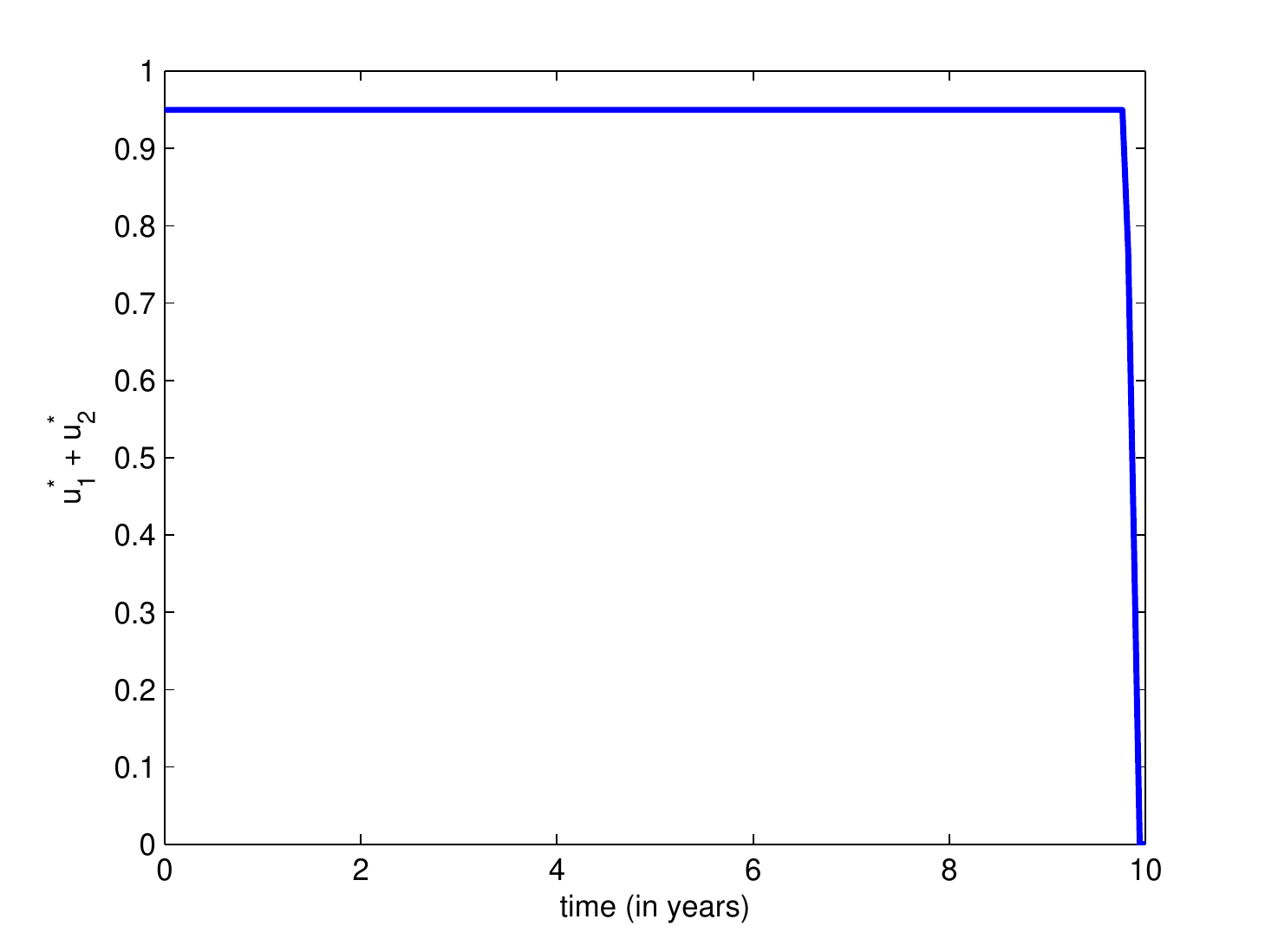}}
\caption{Extremal controls $u_1^*$ and $u_2^*$ for cost functional $J_1$,
with $\beta_1 = 0.6$, $\beta_2 = 0.1$, $W_1 = W_2 = 50$ and parameter values
from Table~\ref{table:parameters:TB-HIV:Chronic}.}
\label{fig:u1:u2:minAeAT:nodeath}
\end{figure}
In this case, we compare the behavior of the dynamics $A$, $A_T$, $I_{TH}$ and $C_H$
for the following cost functionals $J_1$, $J_2$ and $J_3$, with $W_1 = W_2 = 50$, where
\begin{equation}
J_2(u_1(\cdot)) = \int_0^{T} \left[ A(t) + A_T(t)
+ \frac{W_1}{2}u_1^2(t) \right] dt,
\end{equation}
\begin{equation}
J_3(u_2(\cdot)) = \int_0^{T} \left[ A(t) + A_T(t)
+ \frac{W_2}{2}u_2^2(t) \right] dt,
\end{equation}
that is, when both controls $u_1$ and $u_2$
are applied simultaneously, or are applied separately.
\begin{figure}[!htb]
\centering
\subfloat[\footnotesize{$u_1$}]{\label{onlyu1:minAeAT:nodeath}
\includegraphics[width=0.45\textwidth]{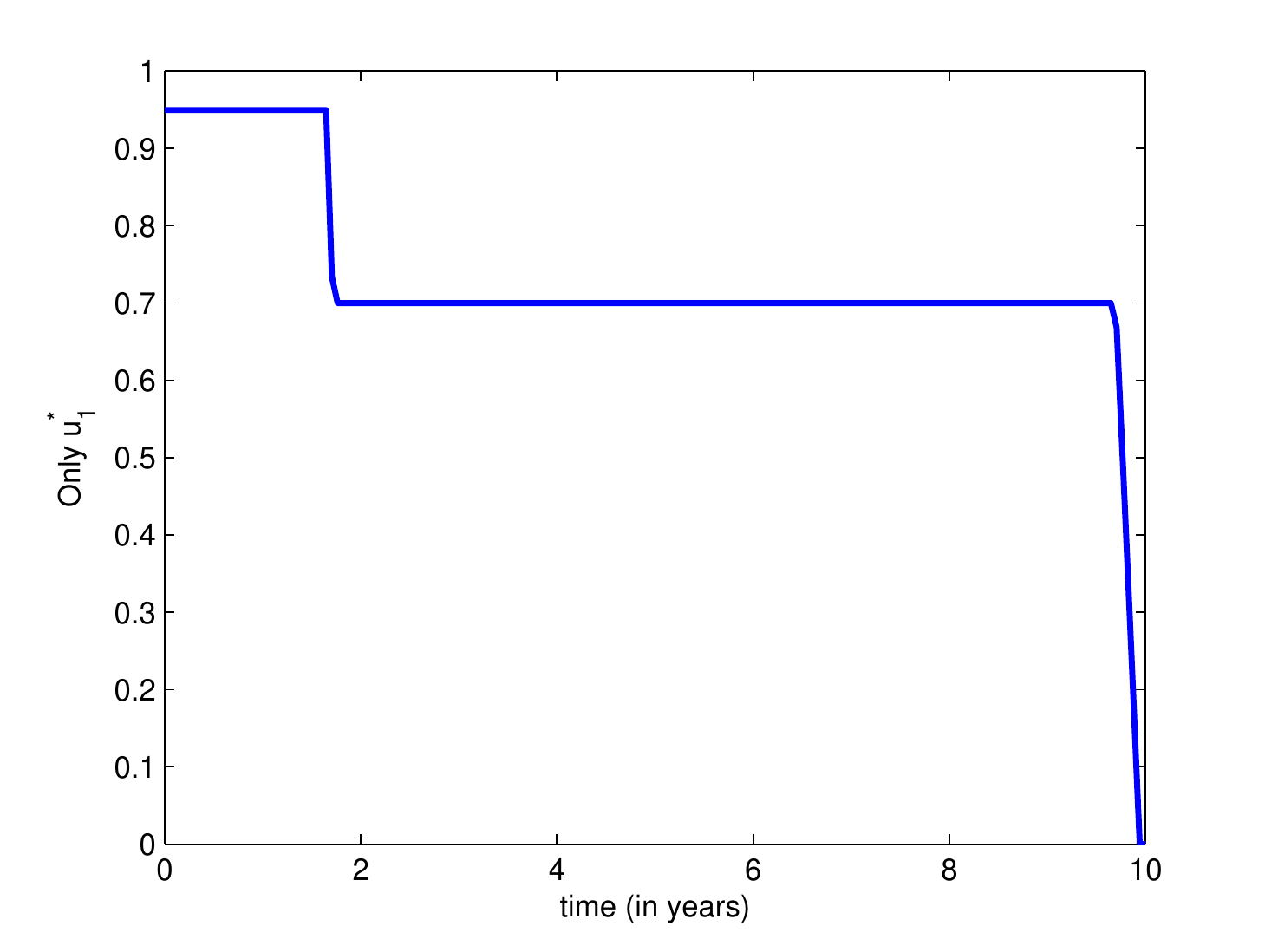}}
\subfloat[\footnotesize{$u_2$}]{\label{onlyu2:minAeAT:nodeath}
\includegraphics[width=0.45\textwidth]{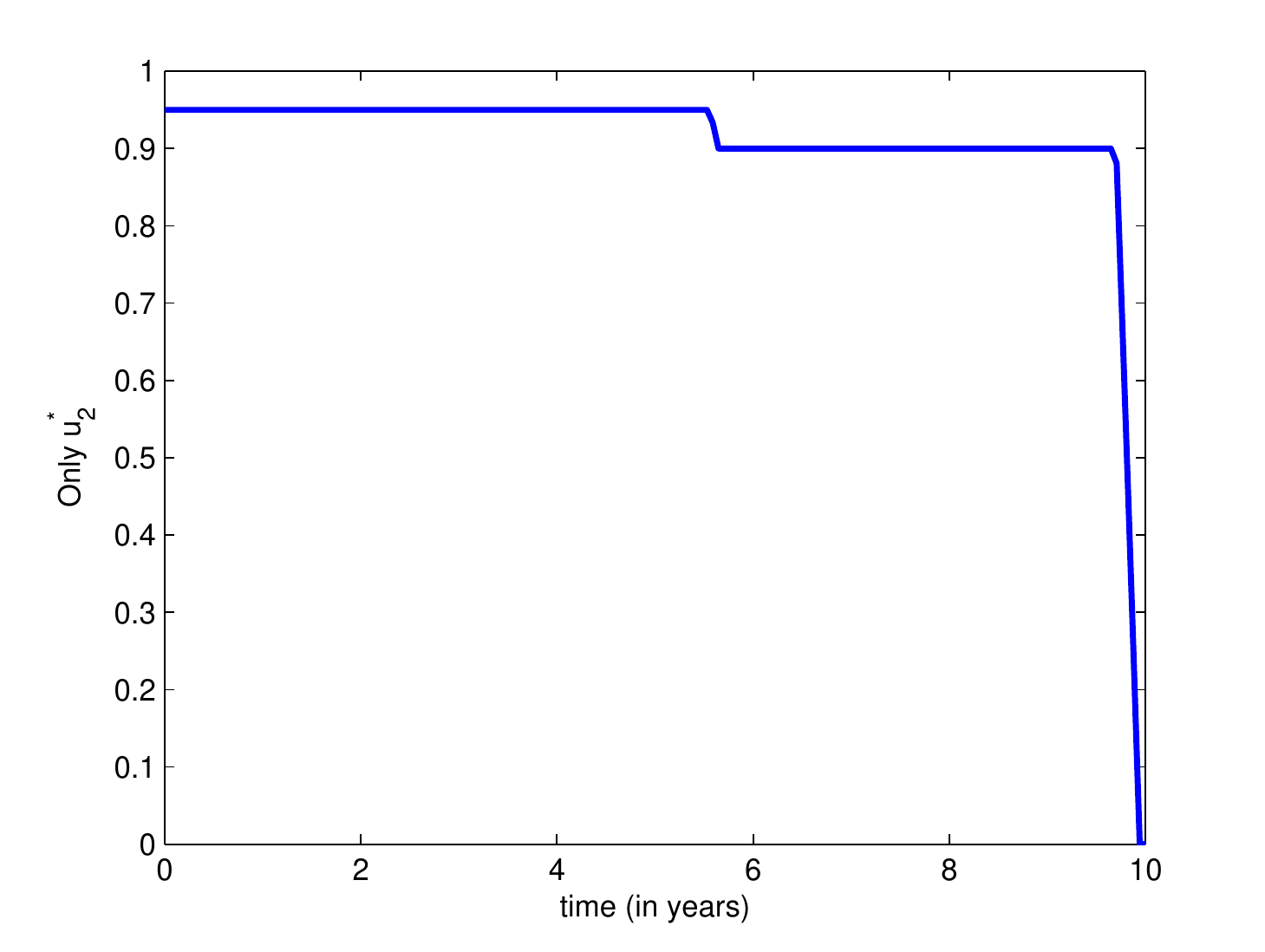}}
\caption{Extremal controls when $u_1$ and $u_2$ are applied separately.}
\label{fig:only:u1:u2:minAeAT:nodeath}
\end{figure}
The number of individuals with only AIDS, $A$, is lower for cost functional $J_1$
and with extremal controls given by $u_1^*$ and $u_2^*$ in
Figure~\ref{fig:u1:u2:minAeAT:nodeath} (a) and (b), see Figure~\ref{fig:A:AT:minAeAT:nodeath} (a).
However, this is not the best strategy for the reduction of the number of individuals
with both AIDS and TB diseases, $A_T$. In this case, the best strategy is to apply only control
$u_1$ where $u_1$ takes the values given in Figure~\ref{fig:only:u1:u2:minAeAT:nodeath},
see Figure~\ref{fig:A:AT:minAeAT:nodeath} (b). The best strategy for the reduction of the
total number of individuals with only AIDS and with both AIDS and TB, $A + A_T$,
during the first six years is the one associated to the cost functional $J_2$, that is,
apply only control $u_1^*$ (see Figure~\ref{fig:only:u1:u2:minAeAT:nodeath} (a)) and
after six years the best strategy is to apply simultaneously both controls $u_1^*$
and $u_2^*$ given by Figure~\ref{fig:u1:u2:minAeAT:nodeath}.
\begin{figure}[!htb]
\centering
\subfloat[\footnotesize{$A$}]{\label{A:minAeAT:nodeath}
\includegraphics[width=0.33\textwidth]{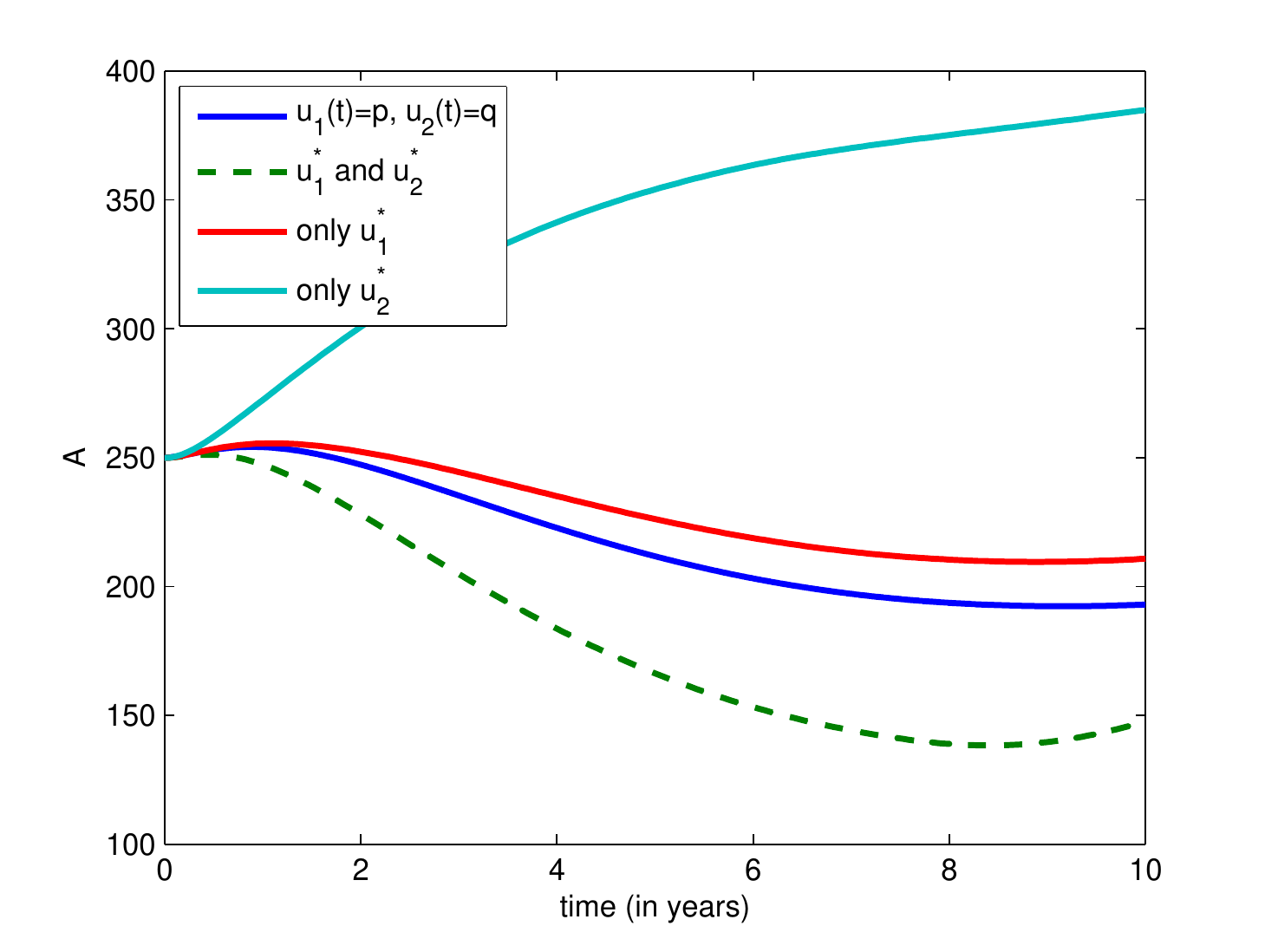}}
\subfloat[\footnotesize{$A_T$}]{\label{AT:minAeAT:nodeath}
\includegraphics[width=0.36\textwidth]{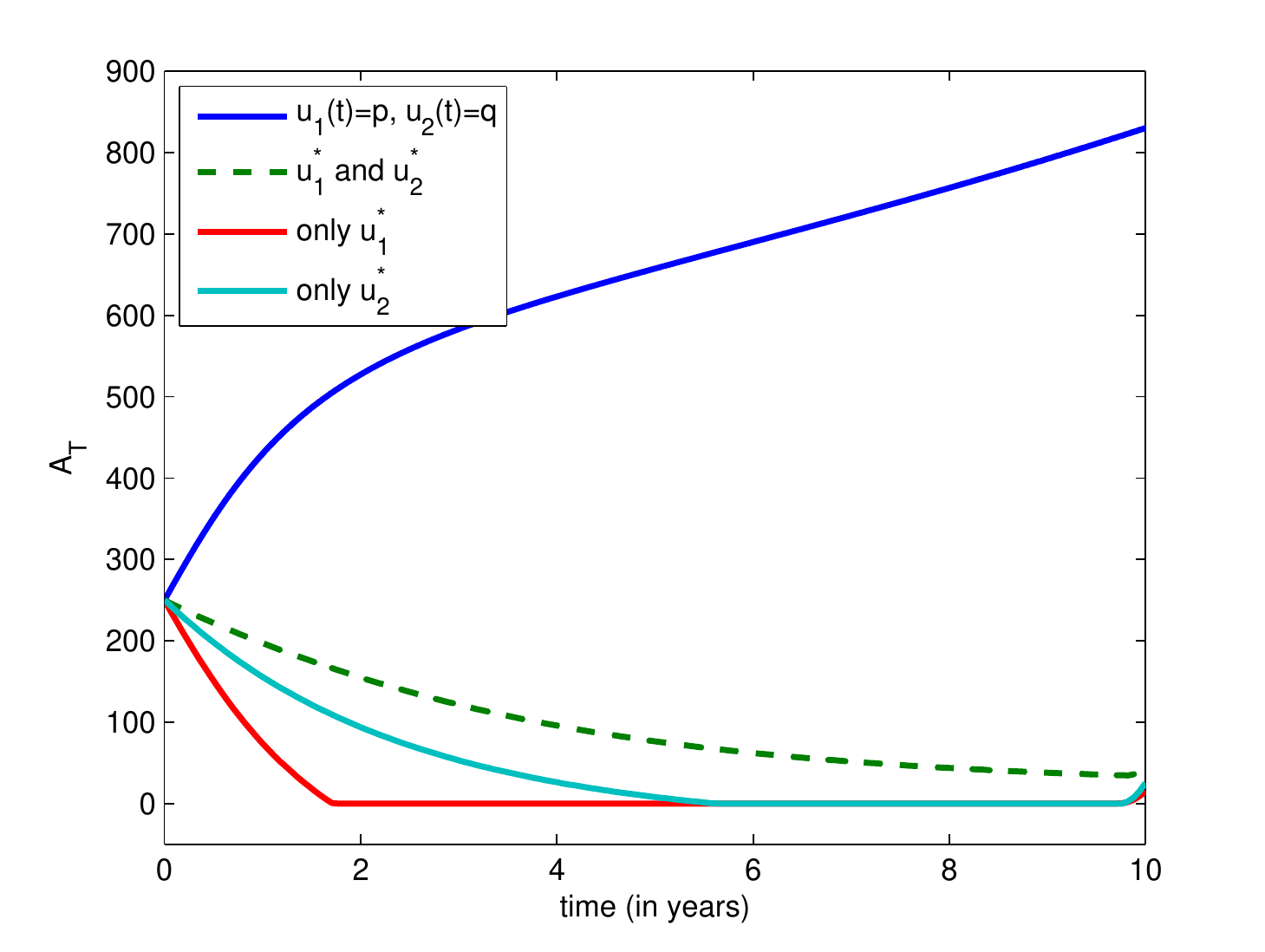}}
\subfloat[\footnotesize{$A + A_T$}]{\label{AmaisAT:minAeAT:nodeath}
\includegraphics[width=0.33\textwidth]{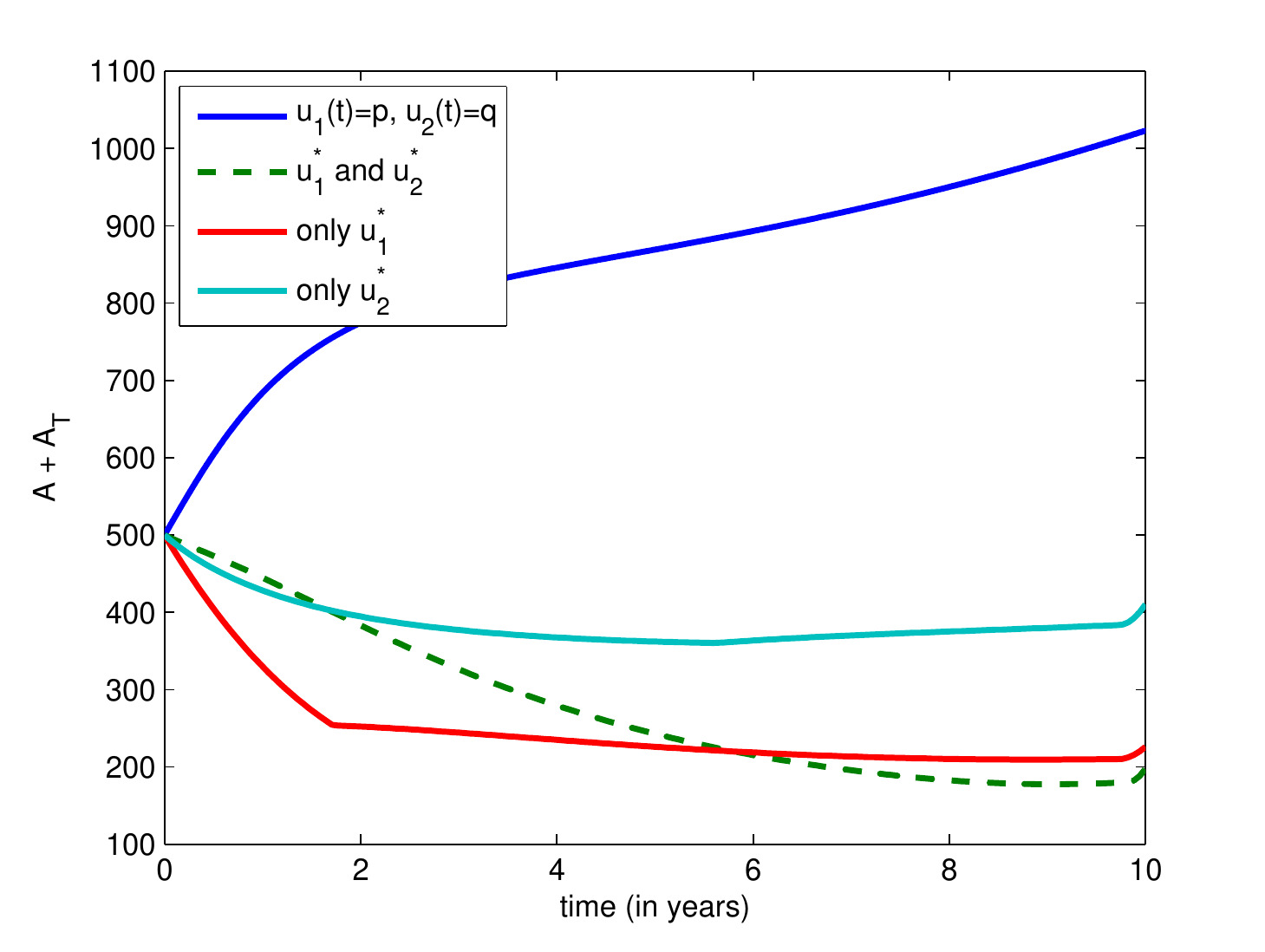}}
\caption{Dynamics $A$, $A_T$ and $A + A_T$ for cost functionals $J_1$,
$J_2$, $J_3$ and $u_1=p=0.1$, $u_2 = q = 0.3$, with $\beta_1 = 0.6$, $\beta_2 = 0.1$,
$W_1 = W_2 = 50$ and parameter values from Table~\ref{table:parameters:TB-HIV:Chronic}.}
\label{fig:A:AT:minAeAT:nodeath}
\end{figure}
In Figure~\ref{fig:ITH:CH:minAeAT:nodeath} we observe that the implementation of controls
$u_1$ and $u_2$ simultaneously or separately, contribute to the reduction of the number of
individuals infected with HIV and TB, $I_{TH}$, and increase the number of individuals
that remain in the class $C_H$, that is, the HIV infection does not evolve to AIDS disease.
In this case, the strategy of treating only TB for the individuals in the class $I_{TH}$
does not allow equal or better results on the reduction of individuals in the class $I_{TH}$,
which happens in the situation described in Figure~\ref{fig:u1:u2:minAT:W1:500}
for the case $W_1 = 500$ and $W_2 = 50$.
\begin{figure}[!htb]
\centering
\subfloat[\footnotesize{$I_{TH}$}]{\label{ITH:minAeAT:nodeath}
\includegraphics[width=0.45\textwidth]{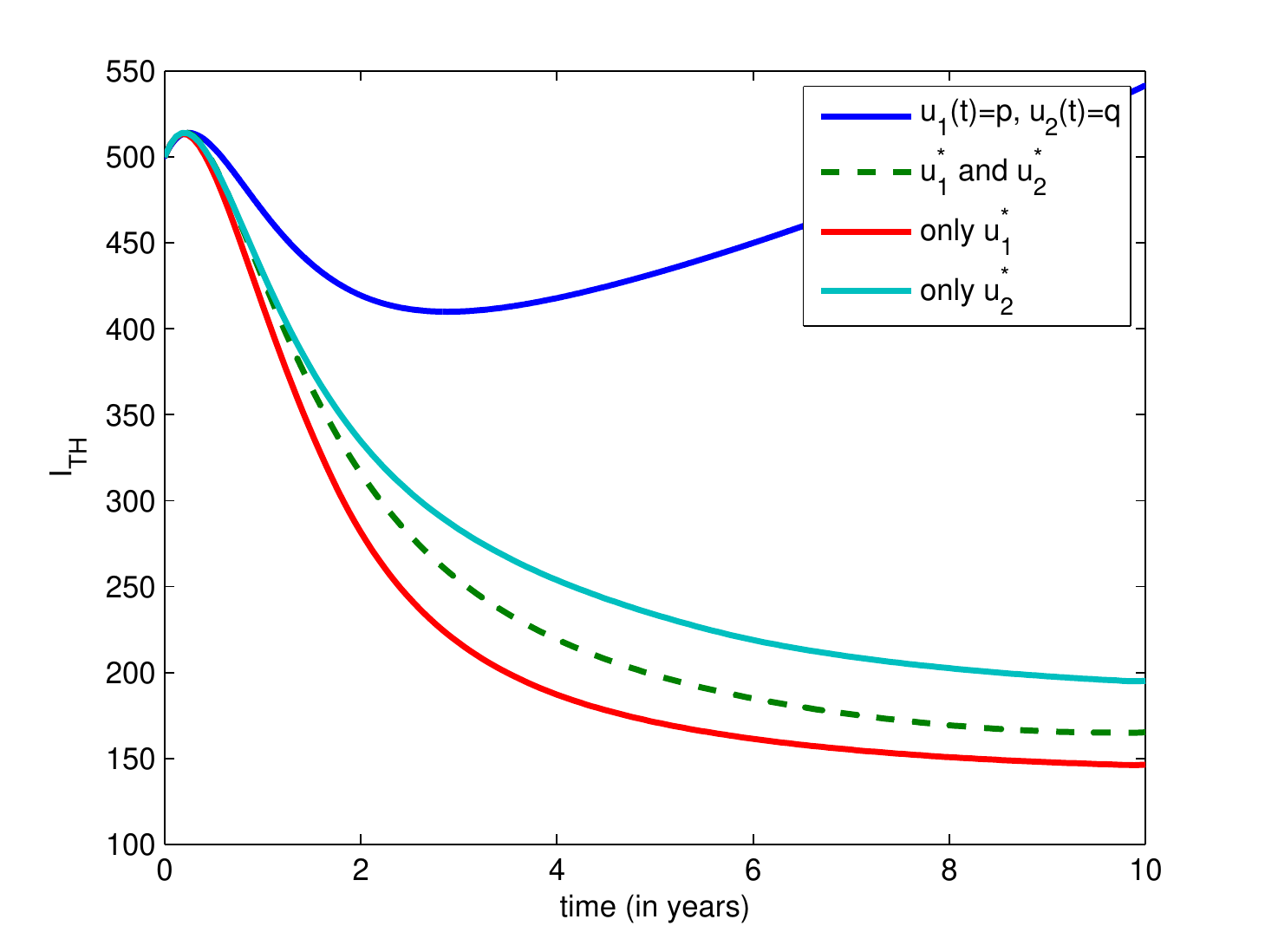}}
\subfloat[\footnotesize{$C_H$}]{\label{CH:minAeAT:nodeath}
\includegraphics[width=0.45\textwidth]{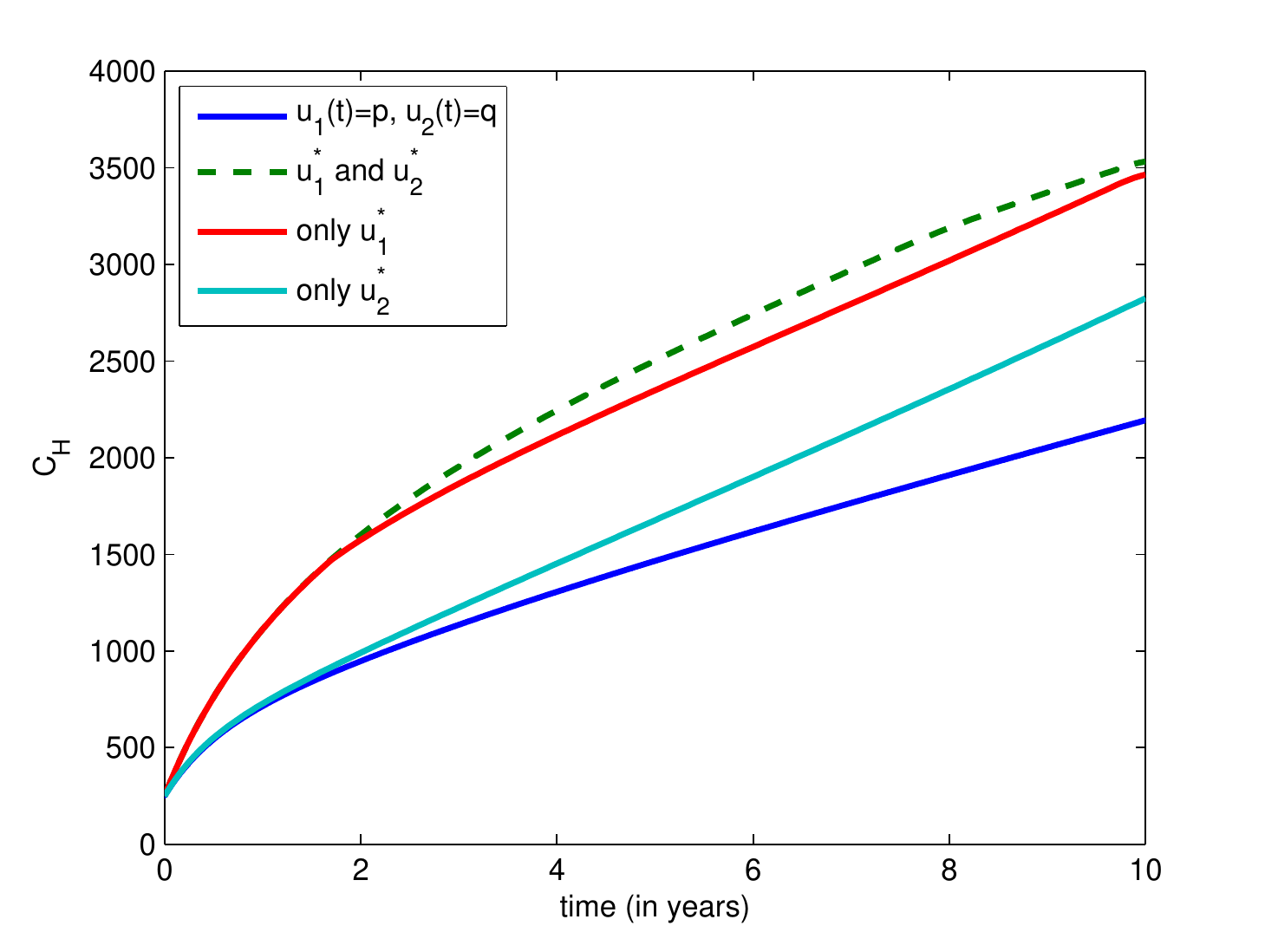}}
\caption{Dynamics $I_{TH}$ and $C_H$ for cost functionals $J_1$, $J_2$, $J_3$
and $u_1=p=0.1$, $u_2 = q = 0.3$,
with $\beta_1 = 0.6$, $\beta_2 = 0.1$, $W_1 = W_2 = 50$ and parameter
values from Table~\ref{table:parameters:TB-HIV:Chronic}.}
\label{fig:ITH:CH:minAeAT:nodeath}
\end{figure}
\begin{table}[!htb]
\centering
\begin{tabular}{l  p{6.5cm} l  l }
\hline \hline
{\small{Symbol}} &  {\small{Description}}   &  {\small{Value}} & {\small{References}}\\
\hline
{\small{$N$}} & {\small{Total population}} & {\small{variable}}  & \\
{\small{$N(0)$}} & {\small{Initial population}} & {\small{$30000$}}  & \\
{\small{$\Lambda$}} & {\small{Recruitment rate}} & {\small{$430$}}  & \\
{\small{$\mu$}} & {\small{Natural death rate}} &  {\small{$1/70 $}} &  \\
{\small{$\beta_1$}} & {\small{TB transmission rate}} &  {\small{variable}} &  \\
{\small{$\beta_2$}} & {\small{HIV transmission rate}} &  {\small{variable}} &  \\
{\small{$\eta_C$}} & {\small{Modification parameter}} &  {\small{$0.9$}} &  \\
{\small{$\eta_A$}} & {\small{Modification parameter}} &  {\small{$1.05$}} &  \\
{\small{$k_1$}} & {\small{Rate at which individuals leave $L_T$ class by becoming infectious}}
&  {\small{$1/2$}} & {\small{\cite{Castillo_Chavez_1997,SLenhart_2002}}}  \\
{\small{$\tau_1$}} & {\small{TB treatment rate for $L_T$ individuals}}
&  {\small{$2 $}} &  {\small{\cite{Castillo_Chavez_1997,SLenhart_2002}}} \\
{\small{$\tau_2$}} & {\small{TB treatment rate for $I_T$ individuals}}
&  {\small{$1 $}} & {\small{\cite{Castillo_Chavez_1997,SLenhart_2002}}}  \\
{\small{$\beta'_1$}} & {\small{Modification parameter}} &  {\small{$0.9$}} &  \\
{\small{$d_T$}} & {\small{TB induced death rate}} &  {\small{$1/10 $}}
& {\small{\cite{Castillo_Chavez_1997}}} \\
{\small{$\delta$}} & {\small{Modification parameter}} &  {\small{$1.03$}} &  \\
{\small{$\psi$}} & {\small{Modification parameter}} &  {\small{$1.07$}} &  \\
{\small{$\phi$}} & {\small{HIV treatment rate for $I_H$ individuals}} &  {\small{$1$}} &  \\
{\small{$\rho_1$}} & {\small{Rate at which individuals leave $I_H$ class to $A$}}
&  {\small{$0.1 $}} &  \\
{\small{$\alpha_1$}} & {\small{AIDS treatment rate}}
&  {\small{$0.33 $}} & {\small{\cite{Bhunu:BMB:2009:HIV:TB}}} \\
{\small{$\omega_1$}} & {\small{Rate at which individuals leave $C_H$ class}}
&  {\small{$0.09$}} &  \\
{\small{$d_A$}} & {\small{AIDS induced death rate}} &  {\small{$0.3 $}} &  \\
{\small{$\rho_2$}} & {\small{Rate at which individuals leave $I_{TH}$ class}}
&  {\small{$1 $}} &  \\
{\small{$p$}} & {\small{Fraction of $I_{TH}$ individuals that take HIV and TB treatment}}
&  {\small{$0.1$}} &  \\
{\small{$q$}} & {\small{Fraction of $I_{TH}$ individuals that take only TB treatment }}
&  {\small{$0.3$}} &  \\
{\small{$\tau_3$}} & {\small{Rate at which individuals leave $L_{TH}$ class}} &  {\small{$2$}} &  \\
{\small{$k_2$}} & {\small{Rate at which individuals leave $L_{TH}$ class by becoming TB infectious}}
&  {\small{$1.3 \, k_1$}} & \\
{\small{$r$}} & {\small{Fraction of $L_{TH}$ individuals that take HIV and TB treatment}} &  {\small{$0.3$}} &  \\
{\small{$\beta'_2$}} & {\small{Modification parameter}} &  {\small{$1.1$}} &  \\
{\small{$\omega_2$}} & {\small{Rate at which individuals leave $R_H$ class}} &  {\small{$0.15$}} &  \\
{\small{$\alpha_2$}} & {\small{HIV treatment rate for $A_T$ individuals}} &  {\small{$0.33 $}} &  \\
{\small{$d_{TA}$}} & {\small{AIDS-TB induced death rate }} &  {\small{$0.33 $}} &  \\
\hline \hline
\end{tabular}
\caption{Parameters of the TB-HIV/AIDS model.}
\label{table:parameters:TB-HIV:Chronic}
\end{table}


\section{Final comments and future work}
\label{sec:FC:fw}

Our numerical results only give extremals and no claims about optimality are made.
As future work, it would be interesting to verify optimality by using second order optimality
conditions and addressing properly the issue of conjugate points. For that, one
needs to extend the theory underlying the computation of conjugate points
and verification of optimality as developed in \cite[Section 5.3]{Book:S_L:12}
and \cite{cit:ref2}.


\section*{Acknowledgements}

This work was partially supported by Portuguese funds through CIDMA
(Center for Research and Development in Mathematics and Applications)
and FCT (The Portuguese Foundation for Science and Technology),
within project PEst-OE/MAT/UI4106/2014.
Silva was also supported by FCT through the post-doc fellowship
SFRH/BPD/72061/2010; Torres by EU funding under
the 7th Framework Programme FP7-PEOPLE-2010-ITN, grant agreement 264735-SADCO;
and by the FCT project OCHERA, PTDC/EEI-AUT/1450/2012, co-financed by
FEDER under POFC-QREN with COMPETE reference FCOMP-01-0124-FEDER-028894.
The authors would like to thank Professor Helmut Maurer from
Institute of Computational and Applied Mathematics, University of Muenster,
Germany, for kindly sharing with them his expertise and for several valuable
comments and helpful suggestions, which improved the quality of the paper;
and to two Referees for several constructive remarks and questions.



\bigskip


\end{document}